\documentclass{amsart}

\usepackage{amssymb}

\usepackage[usenames, svgnames]{xcolor}
\usepackage{tikz}

\theoremstyle{plain}
\newtheorem{theorem}{Theorem}
\newtheorem{lemma}[theorem]{Lemma}
\newtheorem{corollary}[theorem]{Corollary}
\newtheorem{proposition}[theorem]{Proposition}

\theoremstyle{definition}
\newtheorem{definition}[theorem]{Definition}
\newtheorem{example}[theorem]{Example}
\newtheorem{question}[theorem]{Question}

\theoremstyle{remark}
\newtheorem{remark}[theorem]{Remark}
\newtheorem{claim}[theorem]{Claim}
\newtheorem*{acknowledgments}{Acknowledgments}

\tikzset{normal/.style={circle,fill=#1,inner sep=1pt}, special/.style 2 args={circle,fill=#1,draw=#2,inner sep=1.5pt,line width=0.5pt}}

\DeclareMathOperator{\disc}{disc}
\DeclareMathOperator{\ord}{ord}
\DeclareMathOperator{\res}{res}

\begin{document}

\title{Unicritical polynomial maps with rational multipliers}

\author{Valentin Huguin}
\address{Institut de Mathématiques de Toulouse, UMR 5219, Université de Toulouse, CNRS, UPS, F-31062 Toulouse Cedex 9, France}
\email{valentin.huguin@math.univ-toulouse.fr}

\subjclass[2010]{Primary 37P05, 37P35; Secondary 37F10, 37F45}

\begin{abstract}
In this article, we prove that every unicritical polynomial map that has only rational multipliers is either a power map or a Chebyshev map. This provides some evidence in support of a conjecture by Milnor concerning rational maps whose multipliers are all integers.
\end{abstract}

\maketitle

\section{Introduction}

Given a polynomial map $f \colon \mathbb{C} \rightarrow \mathbb{C}$ and a point $z_{0} \in \mathbb{C}$, we study the sequence $\left( f^{\circ n}\left( z_{0} \right) \right)_{n \geq 0}$ of iterates of $f$ at $z_{0}$. The set $\left\lbrace f^{\circ n}\left( z_{0} \right) : n \geq 0 \right\rbrace$ is called the \emph{forward orbit} of $z_{0}$ under $f$.

The point $z_{0}$ is said to be \emph{periodic} for $f$ if there exists an integer $n \geq 1$ such that $f^{\circ n}\left( z_{0} \right) = z_{0}$; the least such integer $n$ is called the \emph{period} of $z_{0}$. The forward orbit of $z_{0}$, which has cardinality $n$, is said to be a \emph{cycle} for $f$. The \emph{multiplier} of $f$ at $z_{0}$ is the derivative of $f^{\circ n}$ at $z_{0}$; equivalently, it is the product of the derivatives of $f$ along the cycle. In particular, $f$ has the same multiplier at each point of the cycle.

The multiplier is invariant under conjugacy: if $f$ and $g$ are two polynomial maps, $\phi$ is an invertible affine map such that $\phi \circ f = g \circ \phi$ and $z_{0}$ is a periodic point for $f$, then $\phi\left( z_{0} \right)$ is a periodic point for $g$ with the same period and the same multiplier.

In this paper, we wish to examine the polynomial maps that have only integer~-- or rational~-- multipliers.

\begin{definition}
A polynomial map $f \colon \mathbb{C} \rightarrow \mathbb{C}$ of degree $d \geq 2$ is said to be a \emph{power map} if it is affinely conjugate to $z \mapsto z^{d}$.
\end{definition}

For every $d \geq 2$, there exists a unique polynomial $T_{d} \in \mathbb{C}[z]$ such that \[ T_{d}\left( z +z^{-1} \right) = z^{d} +z^{-d} \, \text{.} \] The polynomial $T_{d}$ is monic of degree $d$ and is called the $d$th \emph{Chebyshev polynomial}.

\begin{example}
We have $T_{2}(z) = z^{2} -2$ and $T_{3}(z) = z^{3} -3 z$.
\end{example}

\begin{definition}
A polynomial map $f \colon \mathbb{C} \rightarrow \mathbb{C}$ of degree $d \geq 2$ is said to be a \emph{Chebyshev map} if it is affinely conjugate to $\pm T_{d}$.
\end{definition}

\begin{remark}
For every $d \geq 2$, the polynomials $-T_{d}$ and $T_{d}$ are affinely conjugate if and only if $d$ is even.
\end{remark}

These conjugacy classes of polynomials share the following well-known property:

\begin{proposition}
\label{proposition:special}
Suppose that $f \colon \mathbb{C} \rightarrow \mathbb{C}$ is a power map or a Chebyshev map. Then $f$ has only integer multipliers.
\end{proposition}

\subsection{Statement of the results}

We are interested in the converse of Proposition~\ref{proposition:special}. More precisely, we wish to show that every polynomial map that has only integer~-- or rational~-- multipliers is either a power map or a Chebyshev map.

We restrict ourselves to \emph{unicritical} polynomial maps~-- that is, polynomial maps of degree $d \geq 2$ that have a unique critical point in the complex plane.

\begin{theorem}
\label{theorem:unicritical}
Assume that $f \colon \mathbb{C} \rightarrow \mathbb{C}$ is a unicritical polynomial map that has only rational multipliers. Then $f$ is either a power map or a Chebyshev map.
\end{theorem}

\begin{remark}
For every $d \geq 2$, the polynomial $T_{d}$ has exactly $d -1$ critical points given by $2 \cos\left( \frac{\pi j}{d} \right)$ for $j \in \lbrace 1, \dotsc, d -1 \rbrace$. In particular, a Chebyshev map is unicritical if and only if it has degree $2$.
\end{remark}

Using similar arguments, we also obtain a result concerning cubic polynomial maps \emph{with symmetries}~-- that is, cubic polynomial maps that commute with a nontrivial invertible affine map.

\begin{theorem}
\label{theorem:symcubic}
Assume that $f \colon \mathbb{C} \rightarrow \mathbb{C}$ is a cubic polynomial map with symmetries that has only integer multipliers. Then $f$ is a power map or a Chebyshev map.
\end{theorem}

\subsection{Motivation}

In a more general setting, Milnor conjectured in~\cite{M2006} that power maps, Chebyshev maps and flexible Latt\`{e}s maps are the only rational maps whose multipliers are all integers. We may even extend his question as follows:

\begin{question}
Let $K$ be a number field, and denote by $\mathcal{O}_{K}$ its ring of integers. Assume that $f \colon \widehat{\mathbb{C}} \rightarrow \widehat{\mathbb{C}}$ is a rational map whose multipliers all lie in $\mathcal{O}_{K}$~-- or $K$. Is $f$ necessarily a finite quotient of an affine map~-- that is, either a power map, a Chebyshev map or a Latt\`{e}s map?
\end{question}

We give here a positive answer in the case of rational numbers and unicritical polynomial maps. To the author's knowledge, this question has not been studied before. This one can be viewed as an analog of questions concerning rational preperiodic points for a rational map, which have received a lot of attention (see~\cite{BIJMST2019} and~\cite{S2007}).

In~\cite{EvS2011}, Eremenko and van~Strien investigated the rational maps that have only real multipliers: they proved that, if $f \colon \widehat{\mathbb{C}} \rightarrow \widehat{\mathbb{C}}$ is such a map, then either $f$ is a Latt\`{e}s map or its Julia set $\mathcal{J}_{f}$ is contained in a circle; they also gave a description of these maps.

\subsection{Organization of the paper}

In Section~\ref{section:proofs}, we prove some stronger versions of Theorem~\ref{theorem:unicritical}. More precisely, given an integer $d \geq 2$, the conjugacy classes of unicritical polynomials of degree $d$ are parameterized by a one-parameter family $\left( f_{c} \right)_{c \in \mathbb{C}}$, and we determine the parameters $c \in \mathbb{C}$ for which the multiplier polynomials of $f_{c}$ have only rational roots. Using the same strategy, we also prove Theorem~\ref{theorem:symcubic}.

In Section~\ref{section:polynomials}, we study the periodic points and the multipliers of a polynomial map by means of polynomials associated with its dynamics. More precisely, we present certain results about the dynatomic polynomials and the multiplier polynomials of a monic polynomial, emphasizing the case of unicritical polynomials.

Finally, in Section~\ref{section:fermat}, we prove that the Diophantine equation that arose in our proof of Theorem~\ref{theorem:unicritical} has no nontrivial solution.

\begin{acknowledgments}
The author would like to thank his Ph.D. advisors, Xavier Buff and Jasmin Raissy, for all their suggestions and encouragements.
\end{acknowledgments}

\section{Proofs of the results}
\label{section:proofs}

We shall prove here Theorem~\ref{theorem:unicritical} and Theorem~\ref{theorem:symcubic}. Our proofs rely on the result below, which will be presented in greater detail in Section~\ref{section:polynomials} and is an immediate consequence of Proposition~\ref{proposition:dynadef}, Proposition~\ref{proposition:multdef} and Corollary~\ref{corollary:multsplit}.

Fix an integer $d \geq 2$. Since every polynomial map $f \colon \mathbb{C} \rightarrow \mathbb{C}$ is affinely conjugate to a monic polynomial map and the multiplier is invariant under conjugacy, we may restrict our attention to monic polynomials.

\begin{proposition}
\label{proposition:proofs}
Assume that $f \colon \mathbb{C} \rightarrow \mathbb{C}$ is a monic polynomial map of degree $d$. Then 
\begin{itemize}
\item there exists a unique sequence $\left( \Phi_{n}^{f} \right)_{n \geq 1}$ of elements of $\mathbb{C}[z]$ such that, for every $n \geq 1$, we have \[ f^{\circ n}(z) -z = \prod_{k \mid n} \Phi_{k}^{f}(z) \, \text{;} \]
\item for every $n \geq 1$, there is a unique monic polynomial $M_{n}^{f} \in \mathbb{C}[\lambda]$ such that \[ M_{n}^{f}(\lambda)^{n} = \res_{z}\left( \Phi_{n}^{f}(z), \lambda -\left( f^{\circ n} \right)^{\prime}(z) \right) \, \text{,} \] where $\res_{z}$ denotes the resultant with respect to $z$;
\item given a subring $R$ of $\mathbb{C}$ and $n \geq 1$, the multipliers of $f$ at its cycles with period $n$ all lie in $R$ if and only if $M_{n}^{f}$ splits into linear factors of $R[\lambda]$.
\end{itemize}
\end{proposition}

\begin{definition}
Suppose that $f \colon \mathbb{C} \rightarrow \mathbb{C}$ is a monic polynomial map of degree $d$. For $n \geq 1$, the polynomial $\Phi_{n}^{f}$ is called the $n$th \emph{dynatomic polynomial} of $f$ and the polynomial $M_{n}^{f}$ is called the $n$th \emph{multiplier polynomial} of $f$.
\end{definition}

For $c \in \mathbb{C}$, let $f_{c} \colon \mathbb{C} \rightarrow \mathbb{C}$ be the polynomial map \[ f_{c} \colon z \mapsto z^{d} +c \, \text{.} \] For every $c \in \mathbb{C}$, the map $f_{c}$ is unicritical with critical point $0$ and critical value $c$. Furthermore, if $f \colon \mathbb{C} \rightarrow \mathbb{C}$ is a unicritical polynomial map of degree $d$, then there exists a parameter $c \in \mathbb{C}$~-- which is unique up to multiplication by a $(d -1)$th root of unity~-- such that $f$ is affinely conjugate to $f_{c}$.

Consequently, to prove Theorem~\ref{theorem:unicritical}, we are reduced to determining the parameters $c \in \mathbb{C}$ for which the polynomials $M_{n}^{f_{c}} \in \mathbb{C}[\lambda]$, with $n \geq 1$, have only rational roots. Note that, if $c \in \mathbb{C}$ is such a parameter, then, for every $n \geq 1$, the polynomial $M_{n}^{f_{c}}$ lies in $\mathbb{Q}[\lambda]$ and its discriminant \[ \Delta_{n}(c) = \disc M_{n}^{f_{c}} \] is the square of a rational number. In fact, we shall see that, to prove Theorem~\ref{theorem:unicritical}, it suffices to examine the polynomials $M_{n}^{f_{c}}$ for only a few small values of $n$.

\subsection{Quadratic polynomial maps}

Let us examine here the quadratic polynomials that have only integer~-- or rational~-- multipliers.

Suppose that $d = 2$. Then, for every $c \in \mathbb{C}$, the map $f_{c}$ is a power map if and only if $c = 0$ and is a Chebyshev map if and only if $c = -2$. Using the software SageMath, we can compute $M_{n}^{f_{c}}$ and $\Delta_{n}(c)$ for $c \in \mathbb{C}$ and small values of $n$.

\begin{example}
For every $c \in \mathbb{C}$, we have 
\begin{align*}
M_{1}^{f_{c}}(\lambda) & = \lambda^{2} -2 \lambda +4 c \, \text{,}\\
M_{2}^{f_{c}}(\lambda) & = \lambda -4 c -4 \, \text{,}\\
M_{3}^{f_{c}}(\lambda) & = \lambda^{2} +(-8 c -16) \lambda +64 c^{3} +128 c^{2} +64 c +64 \, \text{,}\\
\begin{split}
M_{4}^{f_{c}}(\lambda) & = \lambda^{3} +\left( 16 c^{2} -48 \right) \lambda^{2} +\left( -256 c^{4} -256 c^{3} +256 c^{2} +768 \right) \lambda\\
& \quad -4096 c^{6} -12288 c^{5} -12288 c^{4} -12288 c^{3} -8192 c^{2} -4096 \, \text{.}
\end{split}
\end{align*}
\end{example}

\begin{remark}
\label{remark:multcoeffs2}
Observe that, for $n \in \lbrace 1, \dotsc, 4 \rbrace$, the coefficients of $M_{n}^{f_{c}}$ are polynomials in $4 c$ with integer coefficients. As we shall see in Section~\ref{section:polynomials}, this is true for all $n \geq 1$ (compare~\cite[Lemma~1]{B2014}).
\end{remark}

\begin{example}
For every $c \in \mathbb{C}$, we have 
\begin{align*}
\Delta_{1}(c) & = -2^{2} (4 c -1) \, \text{,}\\
\Delta_{2}(c) & = 1 \, \text{,}\\
\Delta_{3}(c) & = -2^{6} (4 c +7) c^{2} \, \text{,}\\
\Delta_{4}(c) & = -2^{24} \left( 64 c^{3} +144 c^{2} +108 c +135 \right) (c +2)^{2} c^{6} \, \text{.}
\end{align*}
\end{example}

\begin{remark}
\label{remark:deltafact}
We shall see in Section~\ref{section:polynomials} that, for every $n \geq 1$, the roots of $\Delta_{n}$ that have an odd multiplicity are precisely the parameters $c_{0} \in \mathbb{C}$ for which the map $f_{c_{0}}$ has a cycle with period $n$ and multiplier $1$ (see~\cite[Proposition~9]{M1996}).
\end{remark}

First, let us examine the quadratic polynomials whose multipliers are integers. By Proposition~\ref{proposition:proofs}, for every $c \in \mathbb{C}$, the map $f_{c}$ has an integer multiplier at each cycle with period $1$ or $2$ if and only if the polynomials $M_{1}^{f_{c}}$ and $M_{2}^{f_{c}}$ split into linear factors of $\mathbb{Z}[\lambda]$, which occurs if and only if there exists $m \in \mathbb{Z}$ such that $c = \frac{1 -m^{2}}{4}$. In particular, there exist infinitely many such parameters $c \in \mathbb{C}$. In contrast, by considering also the multipliers at the cycles with period $3$, we obtain the following:

\begin{proposition}
Assume that $f \colon \mathbb{C} \rightarrow \mathbb{C}$ is a quadratic polynomial map that has an integer multiplier at each cycle with period less than or equal to $3$. Then $f$ is either a power map or a Chebyshev map.
\end{proposition}

\begin{proof}
There exists a parameter $c \in \mathbb{C}$ such that $f$ is affinely conjugate to $f_{c}$. By Proposition~\ref{proposition:proofs}, the polynomials $M_{n}^{f_{c}}$, with $n \in \lbrace 1, 2, 3 \rbrace$, split into linear factors of $\mathbb{Z}[\lambda]$, and hence $4 c$ is an integer and \[ \Delta_{1}(c) = -2^{2} (4 c -1) \quad \text{and} \quad \Delta_{3}(c) = -2^{6} (4 c +7) c^{2} \] are the squares of integers. Therefore, either $c = 0$ or there exist $a, b \in \mathbb{Z}_{\geq 0}$ such that \[ -(4 c -1) = a^{2} \quad \text{and} \quad -(4 c +7) = b^{2} \, \text{.} \] In the latter case, we have $(a -b) (a +b) = 8$, and hence \[ \left\lbrace \begin{array}{l} a -b = 1\\ a +b = 8 \end{array} \right. \quad \text{or} \quad \left\lbrace \begin{array}{l} a -b = 2\\ a +b = 4 \end{array} \right. \, \text{,} \] which yields $(a, b) = (3, 1)$ and $c = -2$. Thus, the proposition is proved.
\end{proof}

Let us now study the quadratic polynomial maps whose multipliers are rational. There exist infinitely many parameters $c \in \mathbb{C}$ for which the map $f_{c}$ has a rational multiplier at each cycle with period less than or equal to $3$. More precisely, a parameter $c \in \mathbb{C}$ has this property if and only if $c$ is rational and $\Delta_{1}(c)$ and $\Delta_{3}(c)$ are the squares of rational numbers, which occurs if and only if $c = 0$ or there exists $r \in \mathbb{Q}_{\neq 0}$ such that $c = \frac{-\left( r^{4} +3 r^{2} +4 \right)}{4 r^{2}}$. In contrast, by considering also the multipliers at the cycles with period $4$, we are led to examine the rational points on a certain elliptic curve and we obtain the following result, which is a stronger version of Theorem~\ref{theorem:unicritical} in the case of quadratic polynomials:

\begin{proposition}
\label{proposition:quadratic}
Assume that $f \colon \mathbb{C} \rightarrow \mathbb{C}$ is a quadratic polynomial map that has a rational multiplier at each cycle with period less than or equal to $4$. Then $f$ is either a power map or a Chebyshev map.
\end{proposition}

\begin{proof}
There exists a parameter $c \in \mathbb{C}$ such that $f$ is affinely conjugate to $f_{c}$. By Proposition~\ref{proposition:proofs}, the polynomials $M_{n}^{f_{c}}$, with $n \in \lbrace 1, \dotsc, 4 \rbrace$, split into linear factors of $\mathbb{Q}[\lambda]$, and hence $c$ is rational and \[ \Delta_{4}(c) = -2^{24} \left( 64 c^{3} +144 c^{2} +108 c +135 \right) (c +2)^{2} c^{6} \] is the square of a rational number. Note that, if there exists $r \in \mathbb{Q}$ such that \[ -\left( 64 c^{3} +144 c^{2} +108 c +135 \right) = r^{2} \, \text{,} \] then $c \neq \frac{-3}{4}$ and the rational numbers $a = \frac{r -18}{3 (4 c +3)}$ and $b = \frac{-(r +18)}{3 (4 c +3)}$ satisfy \[ a^{3} +b^{3} -4 = \frac{-4 \left( 64 c^{3} +144 c^{2} +108 c +135 +r^{2} \right)}{(4 c +3)^{3}} = 0 \, \text{,} \] which contradicts the fact that the Diophantine equation $x^{3} +y^{3} = 4 z^{3}$ has no solution $(x, y, z) \in \mathbb{Z}^{3}$ with $z \neq 0$ by Lemma~\ref{lemma:fermat}. Therefore, we have $c \in \lbrace -2, 0 \rbrace$. Thus, the proposition is proved.
\end{proof}

\begin{remark}
It follows from~\cite[Theorem~1]{EvS2011} that, for every $c \in \mathbb{C}$, the map $f_{c}$ has a real multiplier at each cycle if and only if $c \in (-\infty, -2] \cup \lbrace 0 \rbrace$. In particular, the property of having only real multipliers does not characterize power maps and Chebyshev maps among the quadratic polynomials.
\end{remark}

\subsection{Unicritical polynomial maps of degree at least $3$}

We shall see here that, unlike in the case of quadratic polynomials, power maps are the only unicritical polynomial maps of degree at least $3$ that have only real multipliers. Note that, for every $c \in \mathbb{C}$, the map $f_{c}$ is a power map if and only if $c = 0$.

First, suppose that $d = 3$. Using the software SageMath, we can compute $M_{n}^{f_{c}}$ and $\Delta_{n}(c)$ for $c \in \mathbb{C}$ and $n \in \lbrace 1, 2 \rbrace$.

\begin{example}
For every $c \in \mathbb{C}$, we have \[ M_{1}^{f_{c}}(\lambda) = \lambda^{3} -6 \lambda^{2} +9 \lambda -27 c^{2} \] and \[ M_{2}^{f_{c}}(\lambda) = \lambda^{3} -27 \lambda^{2} +\left( 162 c^{2} +243 \right) \lambda -729 c^{4} -1458 c^{2} -729 \, \text{.} \]
\end{example}

\begin{remark}
\label{remark:multcoeffs3}
We shall see in Section~\ref{section:polynomials} that, for every $n \geq 1$, the coefficients of $M_{n}^{f_{c}}$ are polynomials in $27 c^{2}$ with integer coefficients (compare~\cite[Theorem~1.1]{M2014}).
\end{remark}

\begin{example}
For every $c \in \mathbb{C}$, we have \[ \Delta_{1}(c) = -3^{6} \left( 27 c^{2} -4 \right) c^{2} \quad \text{and} \quad \Delta_{2}(c) = -3^{12} \left( 27 c^{2} +32 \right) c^{6} \, \text{.} \]
\end{example}

It follows from Proposition~\ref{proposition:proofs} that, for every $c \in \mathbb{C}$, the map $f_{c}$ has a real multiplier at each fixed point if and only if $c^{2}$ is real and $\Delta_{1}(c) \geq 0$, which occurs if and only if $c^{2} \in \left[ 0, \frac{4}{27} \right]$. In particular, power maps are not the only cubic unicritical polynomial maps whose multiplier at each fixed point is real.

\begin{remark}
There also exist infinitely many parameters $c \in \mathbb{C}$ for which the map $f_{c}$ has a rational multiplier at each fixed point. More precisely, a parameter $c \in \mathbb{C}$ has this property if and only if the polynomial $M_{1}^{f_{c}}$ has a rational root and its discriminant $\Delta_{1}(c)$ is the square of a rational number, which occurs if and only if there exists $r \in \mathbb{Q}$ such that $c^{2} = \frac{4 \left( r^{2} -1 \right)^{2}}{\left( r^{2} +3 \right)^{3}}$.
\end{remark}

In contrast, by considering also the multipliers at the cycles with period $2$, we obtain the result below, which immediately implies Theorem~\ref{theorem:unicritical} in the case of cubic unicritical polynomials.

\begin{proposition}
\label{proposition:cubic}
Assume that $f \colon \mathbb{C} \rightarrow \mathbb{C}$ is a cubic unicritical polynomial map that has a real multiplier at each cycle with period $1$ or $2$. Then $f$ is a power map.
\end{proposition}

\begin{proof}
There exists a parameter $c \in \mathbb{C}$ such that $f$ is affinely conjugate to $f_{c}$. By Proposition~\ref{proposition:proofs}, the polynomials $M_{1}^{f_{c}}$ and $M_{2}^{f_{c}}$ split into linear factors of $\mathbb{R}[\lambda]$, and hence $c^{2}$ is real and \[ \Delta_{1}(c) = -3^{6} \left( 27 c^{2} -4 \right) c^{2} \geq 0 \quad \text{and} \quad \Delta_{2}(c) = -3^{12} \left( 27 c^{2} +32 \right) c^{6} \geq 0 \, \text{.} \] Therefore, we have \[ c^{2} \in \left[ \frac{-32}{27}, 0 \right] \cap \left[ 0, \frac{4}{27} \right] = \lbrace 0 \rbrace \, \text{.} \] Thus, the proposition is proved.
\end{proof}

Let us now investigate the unicritical polynomial maps of degree at least $4$ whose multipliers are real. We shall see that, unlike in the case of cubic unicritical polynomials, the property of having a real multiplier at each fixed point characterizes here power maps. Our result relies on the calculation of $M_{1}^{f_{c}}$ for $c \in \mathbb{C}$.

\begin{example}
\label{example:mult1}
Suppose that $d \geq 2$ and $c \in \mathbb{C}$. Then we have \[ M_{1}^{f_{c}}(\lambda) = \res_{z}\left( z^{d} -z +c, \lambda -d z^{d -1} \right) = (-d)^{d} \prod_{j = 1}^{d -1} \left( z_{j}^{d} -z_{j} +c \right) \, \text{,} \] where $z_{1}, \dotsc, z_{d -1}$ are the roots of $d z^{d -1} -\lambda \in \mathbb{C}[z]$. It follows that \[ \begin{split} M_{1}^{f_{c}}(\lambda) & = (-d)^{d} \prod_{j = 1}^{d -1} \left( d^{-1} (\lambda -d) z_{j} +c \right)\\ & = (-d)^{d} \left( c^{d -1} +\sum_{j = 1}^{d -1} d^{-j} (\lambda -d)^{j} \sigma_{j} c^{d -1 -j} \right) \, \text{,} \end{split} \] where $\sigma_{1}, \dotsc, \sigma_{d -1}$ are the elementary symmetric functions of $z_{1}, \dotsc, z_{d -1}$. Therefore, by the relations between roots and coefficients of a polynomial, we have \[ M_{1}^{f_{c}}(\lambda) = \lambda (\lambda -d)^{d -1} +(-d)^{d} c^{d -1} \, \text{.} \]
\end{example}

The following result is a stronger version of Theorem~\ref{theorem:unicritical} in the case of unicritical polynomials of degree at least $4$.

\begin{proposition}
\label{proposition:higher}
Assume that $f \colon \mathbb{C} \rightarrow \mathbb{C}$ is a unicritical polynomial map of degree $d \geq 4$ that has a real multiplier at each fixed point. Then $f$ is a power map.
\end{proposition}

\begin{proof}
There exists a parameter $c \in \mathbb{C}$ such that $f$ is affinely conjugate to $f_{c}$. By Proposition~\ref{proposition:proofs}, the polynomial \[ M_{1}^{f_{c}}(\lambda) = \lambda (\lambda -d)^{d -1} +(-d)^{d} c^{d -1} \] splits into linear factors of $\mathbb{R}[\lambda]$, and hence the same is true of the polynomial \[ L(\lambda) = \lambda^{d} M_{1}^{f_{c}}\left( \lambda^{-1} +d \right) = (-d)^{d} c^{d -1} \lambda^{d} +d \lambda +1 \] and, by Rolle's theorem, of its derivative \[ L^{\prime}(\lambda) = (-1)^{d} d^{d +1} c^{d -1} \lambda^{d -1} +d \, \text{.} \] Therefore, we have $c = 0$ since the set of roots of $L^{\prime}$ is invariant under multiplication by a $(d -1)$th root of unity and $d \geq 4$. Thus, the proposition is proved.
\end{proof}

Finally, we have proved Theorem~\ref{theorem:unicritical}, which follows immediately from Proposition~\ref{proposition:quadratic}, Proposition~\ref{proposition:cubic} and Proposition~\ref{proposition:higher}.

\subsection{Cubic polynomial maps with symmetries}

We shall use here the same strategy to study the cubic polynomial maps with symmetries whose multipliers are integers and prove Theorem~\ref{theorem:symcubic}.

For $a \in \mathbb{C}$, let $g_{a} \colon \mathbb{C} \rightarrow \mathbb{C}$ be the cubic polynomial map \[ g_{a} \colon z \mapsto z^{3} +a z \, \text{.} \] For every $a \in \mathbb{C}$, the map $g_{a}$ fixes $0$ with multiplier $a$ and commutes with $z \mapsto -z$. Furthermore, if $f \colon \mathbb{C} \rightarrow \mathbb{C}$ is a cubic polynomial map with symmetries, then there exists a unique parameter $a \in \mathbb{C}$ such that $f$ is affinely conjugate to $g_{a}$.

Unlike the family of cubic unicritical polynomial maps, the family of cubic polynomial maps with symmetries contains both power maps and Chebyshev maps. More precisely, for every $a \in \mathbb{C}$, the map $g_{a}$ is a power map if and only if $a = 0$ and is a Chebyshev map if and only if $a = \pm 3$.

Using the software SageMath, we can compute $M_{n}^{g_{a}}$ for $a \in \mathbb{C}$ and $n \in \lbrace 1, 2, 3 \rbrace$.

\begin{example}
For every $a \in \mathbb{C}$, we have 
\begin{align*}
M_{1}^{g_{a}}(\lambda) & = (\lambda -a) (\lambda +2 a -3)^{2} \, \text{,}\\
M_{2}^{g_{a}}(\lambda) & = \left( \lambda -4 a^{2} -12 a -9 \right) \left( \lambda +2 a^{2} -9 \right)^{2} \, \text{,}\\
M_{3}^{g_{a}}(\lambda) & = N_{3}(a, \lambda)^{2} \, \text{,}
\end{align*} 
where $N_{3} \in \mathbb{Z}[a, \lambda]$ is given by \[ \begin{split} N_{3}(a, \lambda) & = \lambda^{4} +\bigl( 2 a^{3} +12 a^{2} -18 a -108 \bigr) \lambda^{3}\\ & \quad +\bigl( -48 a^{6} -72 a^{5} +396 a^{4} +486 a^{3} -324 a^{2} +1458 a +4374 \bigr) \lambda^{2}\\ & \quad +\bigl( 32 a^{9} -792 a^{7} -432 a^{6} +5832 a^{5} +5832 a^{4} -7290 a^{3} -8748 a^{2}\\ & \quad -39366 a -78732 \bigr) \lambda +256 a^{12} +384 a^{11} -4608 a^{10} -6912 a^{9}\\ & \quad +24624 a^{8} +36936 a^{7} -23328 a^{6} -34992 a^{5} -131220 a^{4}\\ & \quad -196830 a^{3} +236196 a^{2} +354294 a +531441 \, \text{.} \end{split} \] Moreover, we have \[ \disc_{\lambda} N_{3}(a, \lambda) = 2^{12} 3^{12} D_{3}(a) \left( 4 a^{3} +12 a^{2} -3 a -27 \right)^{2} (a -3)^{4} (a +3)^{4} a^{12} \, \text{,} \] where $\disc_{\lambda}$ denotes the discriminant with respect to $\lambda$ and $D_{3} \in \mathbb{Z}[a]$ is given by \[ D_{3}(a) = 4 a^{8} +16 a^{7} -35 a^{6} -206 a^{5} -113 a^{4} +376 a^{3} +715 a^{2} +1690 a +2197 \, \text{.} \]
\end{example}

It follows from Proposition~\ref{proposition:proofs} that, for every $a \in \mathbb{C}$, the map $g_{a}$ has an integer multiplier at each cycle with period $1$ or $2$ if and only if $a$ is an integer. By considering also the multipliers at the cycles with period $3$, we obtain the following stronger version of Theorem~\ref{theorem:symcubic}:

\begin{proposition}
Assume that $f \colon \mathbb{C} \rightarrow \mathbb{C}$ is a cubic polynomial map with symmetries that has an integer multiplier at each cycle with period less than or equal to $3$. Then $f$ is either a power map or a Chebyshev map.
\end{proposition}

\begin{proof}
There exists a parameter $a \in \mathbb{C}$ such that $f$ is affinely conjugate to $g_{a}$. By Proposition~\ref{proposition:proofs}, the polynomials $M_{n}^{g_{a}}$, with $n \in \lbrace 1, 2, 3 \rbrace$, split into linear factors of $\mathbb{Z}[\lambda]$, and hence $a$ is an integer and \[ \disc_{\lambda} N_{3}(a, \lambda) = 2^{12} 3^{12} D_{3}(a) \left( 4 a^{3} +12 a^{2} -3 a -27 \right)^{2} (a -3)^{4} (a +3)^{4} a^{12} \] is the square of an integer. Now, note that, if \[ D_{3}(a) = 4 a^{8} +16 a^{7} -35 a^{6} -206 a^{5} -113 a^{4} +376 a^{3} +715 a^{2} +1690 a +2197 \] is the square of an integer, then its residue class in $\mathbb{Z}/32 \mathbb{Z}$ is a square, and hence $a \equiv 1 \pmod{8}$. Moreover, observe that $D_{3}(1 +8 b)$ is not the square of an integer whenever $b \in \lbrace -7, \dotsc, 13 \rbrace$ and we have \[ L(b)^{2} < D_{3}(1 +8 b) < \left( L(b) +1 \right)^{2} \] for all $b \in \mathbb{Z} \setminus \lbrace -7, \dotsc, 13 \rbrace$, where \[ L(b) = 8192 a^{4} +6144 a^{3} +720 a^{2} -252 a -50 \, \text{.} \] Therefore, $D_{3}(a)$ is not the square of an integer, and hence $a \in \lbrace -3, 0, 3 \rbrace$. Thus, the proposition is proved.
\end{proof}

Using the software SageMath, we obtain that $D_{3}(a)$ is not the square of a rational number whenever $a$ is a rational number with height at most $10^{4}$. Thus, it seems likely that the question below has a negative answer, which would imply that every cubic polynomial map with symmetries that has a rational multiplier at each cycle with period less than or equal to $3$ is either a power map or a Chebyshev map.

\begin{question}
Does the hyperelliptic curve of genus $3$ over $\mathbb{Q}$ given by $b^{2} = D_{3}(a)$ have a rational point other than the two points at infinity?
\end{question}

\begin{remark}
Note that the curve of genus $1$ given by $N_{3}(a, \lambda) = 0$ together with the point $\left( \frac{9}{2}, \frac{1647}{4} \right)$ defines an elliptic curve $E$ over $\mathbb{Q}$. Using the software Magma, we obtain that its group of rational points $E(\mathbb{Q})$ is a free abelian group of rank $1$. In particular, there exist infinitely many parameters $a \in \mathbb{C}$ for which the map $g_{a}$ has a rational multiplier at each cycle with period $1$ or $2$ and at a cycle with period $3$. Another approach to proving that power maps and Chebyshev maps are the only cubic polynomial maps with symmetries that have a rational multiplier at each cycle with period less than or equal to $3$ could be to show that the group $E(\mathbb{Q})$ does not contain $4$ distinct points with the same $a$-coordinate.
\end{remark}

\section{Certain polynomials related to the dynamics of a polynomial map}
\label{section:polynomials}

We study here the dynamics of a polynomial map from an algebraic point of view. More precisely, we present certain results about the dynatomic polynomials and the multiplier polynomials of a monic polynomial. In particular, we shall prove here Proposition~\ref{proposition:proofs}, which is the key point in our proofs of Theorem~\ref{theorem:unicritical} and Theorem~\ref{theorem:symcubic}. The results presented in this section are mostly well known.

Fix an integer $d \geq 2$. Although we are interested in the dynamics of complex polynomial maps, we shall consider monic polynomials over an arbitrary integral domain in order to derive information about the coefficients of the dynatomic polynomials and the multiplier polynomials.

\subsection{Dynatomic polynomials}

Let us present here the dynatomic polynomials of a monic polynomial, which are related to its periodic points.

If $R$ is a commutative ring and $f \in R[z]$ is a monic polynomial of degree $d$, then, for every $n \geq 1$, the roots of the polynomial $f^{\circ n}(z) -z$ are precisely the periodic points for the map $f \colon R \rightarrow R$ with period dividing $n$. Thus, it is natural to try to factor these polynomials in order to separate the periodic points for $f \colon R \rightarrow R$ according to their periods.

In the particular case where $a_{0}, \dotsc, a_{d -1}$ are indeterminates over $\mathbb{Z}$, \[ R = \mathbb{Z}\left[ a_{0}, \dotsc, a_{d -1} \right] \quad \text{and} \quad f(z) = z^{d} +\sum_{j = 0}^{d -1} a_{j} z^{j} \in R[z] \, \text{,} \] the polynomials $f^{\circ n}(z) -z$, with $n \geq 1$, are separable, which allows us to factor them over $R$ according to the periods of their roots. By a specialization argument, this provides a factorization of the polynomials $f^{\circ n}(z) -z$, with $n \geq 1$, in the general case of a commutative ring $R$ and a monic polynomial $f \in R[z]$ of degree $d$ (compare~\cite[Section~2]{MP1994}). More precisely, we have the following result, where $\mu \colon \mathbb{Z}_{\geq 1} \rightarrow \lbrace -1, 0, 1 \rbrace$ denotes the M\"{o}bius function:

\begin{proposition}
\label{proposition:dynadef}
Suppose that $R$ is a commutative ring and $f \in R[z]$ is a monic polynomial of degree $d$. Then there exists a unique sequence $\left( \Phi_{n}^{f} \right)_{n \geq 1}$ of elements of $R[z]$ such that, for every $n \geq 1$, we have \[ f^{\circ n}(z) -z = \prod_{k \mid n} \Phi_{k}^{f}(z) \, \text{.} \] Furthermore, for every $n \geq 1$, the polynomial $\Phi_{n}^{f}$ is monic of degree $\nu(n)$, where \[ \nu(n) = \sum_{k \mid n} \mu\left( \frac{n}{k} \right) d^{k} \, \text{.} \]
\end{proposition}

\begin{definition}
Suppose that $R$ is a commutative ring and $f \in R[z]$ is a monic polynomial of degree $d$. For $n \geq 1$, the polynomial $\Phi_{n}^{f}$ is called the $n$th \emph{dynatomic polynomial} of $f$.
\end{definition}

\begin{remark}
If $R$ is an integral domain and $f \in R[z]$ is a monic polynomial of degree $d$, then it follows from the M\"{o}bius inversion formula that, for every $n \geq 1$, we have \[ \Phi_{n}^{f}(z) = \prod_{k \mid n} \left( f^{\circ k}(z) -z \right)^{\mu\left( \frac{n}{k} \right)} \, \text{.} \]
\end{remark}

\begin{remark}
If $R$ is a commutative ring and $f \in R[z]$ is a nonmonic polynomial of degree $d$, then the existence of a sequence $\left( \Phi_{n}^{f} \right)_{n \geq 1}$ as in Proposition~\ref{proposition:dynadef} holds but the uniqueness may fail when $R$ is not an integral domain. For example, if \[ R = \mathbb{Z}/4 \mathbb{Z} \quad \text{and} \quad f(z) = 2 z^{2} +z \in R[z] \, \text{,} \] then we have \[ f^{\circ n}(z) = \begin{cases} 2 z^{2} +z & \text{if } n \text{ is odd}\\ z & \text{if } n \text{ is even} \end{cases} \] for all $n \geq 1$, and hence every sequence $\left( \Phi_{n}^{f} \right)_{n \geq 1}$ of elements of $R[z]$ satisfying \[ \Phi_{1}^{f}(z) = 2 z^{2} \, \text{,} \quad \Phi_{2}^{f} \in 2 R[z] \quad \text{and} \quad \Phi_{n}^{f}(z) = 1 \quad \text{for} \quad n \geq 3 \] is such that, for every $n \geq 1$, we have \[ f^{\circ n}(z) -z = \prod_{k \mid n} \Phi_{k}^{f}(z) \, \text{.} \]
\end{remark}

Let us now describe the roots of the dynatomic polynomials. If $R$ is a commutative ring, $f \in R[z]$ is a monic polynomial of degree $d$ and $n \geq 1$, then each root of the polynomial $\Phi_{n}^{f}$ is a periodic point for the map $f \colon R \rightarrow R$ with period dividing $n$ since it is also a root of the polynomial $f^{\circ n}(z) -z$. Conversely, if $R$ is an integral domain, then each periodic point for $f \colon R \rightarrow R$ with period $n$ is a root of $\Phi_{n}^{f}$. However, it may occur that roots of $\Phi_{n}^{f}$ have period less than $n$. More precisely, we have the following result:

\begin{proposition}[{\cite[Proposition~3.2]{MS1995}}]
\label{proposition:dynaroots}
Assume that $R$ is an integral domain, $f \in R[z]$ is a monic polynomial of degree $d$ and $n \geq 1$. Then $z_{0} \in R$ is a root of the polynomial $\Phi_{n}^{f}$ if and only if $z_{0}$ is a periodic point for $f \colon R \rightarrow R$ with period $k \geq 1$ and multiplier $\lambda_{0} \in R$ that satisfy 
\begin{itemize}
\item $k = n$,
\item or there exists an integer $l \geq 1$ such that $n = k l$ and $\lambda_{0}$ is a primitive $l$th root of unity,
\item or $R$ has characteristic $p > 0$ and there exist integers $l, m \geq 1$ such that $n = k l p^{m}$ and $\lambda_{0}$ is a primitive $l$th root of unity.
\end{itemize}
\end{proposition}

\begin{remark}
If $R$ is a commutative ring that is not an integral domain, $f \in R[z]$ is a monic polynomial of degree $d$ and $n \geq 1$, then it may occur that periodic points for $f \colon R \rightarrow R$ with period $n$ are not roots of the polynomial $\Phi_{n}^{f}$. For example, if \[ R = \mathbb{Z}/4 \mathbb{Z} \quad \text{and} \quad f(z) = z^{2} +3 z +2 \in R[z] \, \text{,} \] then $0 \in R$ is a periodic point for $f$ with period $2$ but is not a root of the polynomial \[ \Phi_{2}^{f}(z) = z^{2} +2 \in R[z] \, \text{.} \]
\end{remark}

Now, let us focus on the family $\left( f_{c} \right)_{c \in \mathbb{C}}$ of unicritical polynomial maps given by \[ f_{c}(z) = z^{d} +c \, \text{.} \]

Given a parameter $c \in \mathbb{C}$, every cycle of parabolic basins for $f_{c}$ contains the unique critical point $0 \in \mathbb{C}$, which allows us to determine the multiplicities of the roots of the polynomials $\Phi_{n}^{f_{c}}$, with $n \geq 1$. Thus, we have the result below, which is a stronger version of Proposition~\ref{proposition:dynaroots} in the case of the polynomial $f_{c}$.

\begin{proposition}[{\cite[Proposition~2.2]{BL2014}}]
\label{proposition:dynarootsuni}
Assume that $c \in \mathbb{C}$ and $n \geq 1$. Then $z_{0} \in \mathbb{C}$ is a root of the polynomial $\Phi_{n}^{f_{c}}$ if and only if 
\begin{itemize}
\item either $z_{0}$ is a periodic point for $f_{c}$ with period $n$ and multiplier different from $1$, in which case $\ord_{z_{0}} \Phi_{n}^{f_{c}} = 1$,
\item or $z_{0}$ is a periodic point for $f_{c}$ with period $n$ and multiplier $1$, in which case $\ord_{z_{0}} \Phi_{n}^{f_{c}} = 2$,
\item or $z_{0}$ is a periodic point for $f_{c}$ with period a proper divisor $k$ of $n$ and multiplier a primitive $\frac{n}{k}$th root of unity, in which case $\ord_{z_{0}} \Phi_{n}^{f_{c}} = \frac{n}{k}$.
\end{itemize}
\end{proposition}

Now, view $c$ as an indeterminate over $\mathbb{Z}$ and consider \[ \boldsymbol{R} = \mathbb{Z}[c] \quad \text{and} \quad \boldsymbol{f}(z) = z^{d} +c \in \boldsymbol{R}[z] \, \text{.} \] For $n \geq 1$, define $\Phi_{n} \in \mathbb{Z}[c, z]$ to be the image of $\Phi_{n}^{\boldsymbol{f}}$ under the canonical ring isomorphism from $\boldsymbol{R}[z]$ to $\mathbb{Z}[c, z]$.

By the uniqueness in Proposition~\ref{proposition:dynadef}, for every $c \in \mathbb{C}$ and every $n \geq 1$, we have \[ \Phi_{n}^{f_{c}}(z) = \Phi_{n}(c, z) \in \mathbb{C}[z] \, \text{.} \] In particular, the coefficients of the polynomials $\Phi_{n}^{f_{c}}$, with $n \geq 1$, are polynomials in $c$ with integer coefficients.

Finally, let us state the result below due to Bousch. It has also been proved with different approaches by Buff and Tan (see~\cite[Theorem~1.2]{BL2014}), Morton (see~\cite[Corollary~1]{M1996}) and Schleicher (see~\cite[Theorem~7.1]{S2017}).

\begin{proposition}[{\cite[Chapitre~3, Th\'{e}or\`{e}me~1]{B1992}}]
\label{proposition:dynairr}
For every $n \geq 1$, the polynomial $\Phi_{n}$ is irreducible over $\mathbb{C}$.
\end{proposition}

\subsection{Multiplier polynomials}

We shall now present the multiplier polynomials of a monic polynomial, which play a crucial role in our proofs of Theorem~\ref{theorem:unicritical} and Theorem~\ref{theorem:symcubic}.

Note that, if $a_{0}, \dotsc, a_{d -1}$ are indeterminates over $\mathbb{Z}$, \[ R = \mathbb{Z}\left[ a_{0}, \dotsc, a_{d -1} \right] \quad \text{and} \quad f(z) = z^{d} +\sum_{j = 0}^{d -1} a_{j} z^{j} \in R[z] \, \text{,} \] then, for every $n \geq 1$, we have \[ \res_{z}\left( \Phi_{n}^{f}(z), \lambda -\left( f^{\circ n} \right)^{\prime}(z) \right) = \prod_{j = 1}^{\nu(n)} \left( \lambda -\left( f^{\circ n} \right)^{\prime}\left( z_{j} \right) \right) \, \text{,} \] where $z_{1}, \dotsc, z_{\nu(n)}$ are the roots of the polynomial $\Phi_{n}^{f}$ in an algebraic closure $K$ of the field of fractions of $R$. Since the map $f \colon K \rightarrow K$ has the same multiplier at each point of a cycle and the roots of $\Phi_{n}^{f}$ in $K$ are simple and are precisely the periodic points for $f \colon K \rightarrow K$ with period $n$, it follows that the polynomial \[ \res_{z}\left( \Phi_{n}^{f}(z), \lambda -\left( f^{\circ n} \right)^{\prime}(z) \right) \in R[\lambda] \] is the $n$th power of some monic polynomial in $R[\lambda]$. By a specialization argument, the same is true in the general case of a commutative ring $R$ and a monic polynomial $f \in R[z]$ of degree $d$ (compare~\cite[Section~5]{MP1994}). More precisely, we have the following result:

\begin{proposition}
\label{proposition:multdef}
Suppose that $R$ is an integral domain and $f \in R[z]$ is a monic polynomial of degree $d$. Then, for every $n \geq 1$, there exists a unique monic polynomial $M_{n}^{f} \in R[\lambda]$ that satisfies \[ M_{n}^{f}(\lambda)^{n} = \res_{z}\left( \Phi_{n}^{f}(z), \lambda -\left( f^{\circ n} \right)^{\prime}(z) \right) \, \text{.} \] Furthermore, the polynomial $M_{n}^{f}$ has degree $\frac{\nu(n)}{n}$.
\end{proposition}

\begin{definition}
Suppose that $R$ is an integral domain and $f \in R[z]$ is a monic polynomial of degree $d$. For $n \geq 1$, the polynomial $M_{n}^{f}$ is called the $n$th \emph{multiplier polynomial} of $f$.
\end{definition}

\begin{remark}
If $R$ is a commutative ring that is not an integral domain and $f \in R[z]$ is a monic polynomial of degree $d$, then, for every $n \geq 1$, the existence of a polynomial $M_{n}^{f}$ as in Proposition~\ref{proposition:multdef} holds but the uniqueness may fail. For example, if \[ R = \mathbb{Z}/4 \mathbb{Z} \quad \text{and} \quad f(z) = z^{2} \in R[z] \, \text{,} \] then we have $\left( f^{\circ 2} \right)^{\prime}(z) = 0$, and hence \[ \res_{z}\left( \Phi_{2}^{f}(z), \lambda -\left( f^{\circ 2} \right)^{\prime}(z) \right) = \lambda^{2} = (\lambda +2)^{2} \, \text{.} \]
\end{remark}

If $K$ is an algebraically closed field, $f \in K[z]$ is a monic polynomial of degree $d$ and $n \geq 1$, then we have \[ M_{n}^{f}(\lambda)^{n} = \prod_{j = 1}^{\nu(n)} \left( \lambda -\left( f^{\circ n} \right)^{\prime}\left( z_{j} \right) \right) \, \text{,} \] where $z_{1}, \dotsc, z_{\nu(n)}$ are the~-- not necessarily distinct~-- roots of the polynomial $\Phi_{n}^{f}$. Note that, if $z_{0} \in K$ is a periodic point for $f$ with period a proper divisor $k$ of $n$ and multiplier a $\frac{n}{k}$th root of unity, then we have $\left( f^{\circ n} \right)^{\prime}\left( z_{0} \right) = 1$ by the chain rule. Therefore, by Proposition~\ref{proposition:dynaroots}, we have the result below, which gives the connection between the multipliers of a monic polynomial map and its multiplier polynomials.

\begin{proposition}
\label{proposition:multroots}
Assume that $K$ is an algebraically closed field, $f \in K[z]$ is a monic polynomial of degree $d$ and $n \geq 1$. Then $\lambda_{0} \in K$ is a root of the polynomial $M_{n}^{f}$ if and only if 
\begin{itemize}
\item $\lambda_{0}$ is the multiplier of $f$ at a cycle with period $n$,
\item or $\lambda_{0}$ equals $1$ and there exist integers $k, l \geq 1$ such that $n = k l$ and $f$ has a cycle with period $k$ and multiplier a primitive $l$th root of unity,
\item or $\lambda_{0}$ equals $1$, the field $K$ has characteristic $p > 0$ and there exist integers $k, l, m \geq 1$ such that $n = k l p^{m}$ and $f$ has a cycle with period $k$ and multiplier a primitive $l$th root of unity.
\end{itemize}
\end{proposition}

An immediate consequence of Proposition~\ref{proposition:multroots} is the following result, which is the key ingredient in our proofs of Theorem~\ref{theorem:unicritical} and Theorem~\ref{theorem:symcubic}:

\begin{corollary}
\label{corollary:multsplit}
Assume that $K$ is an algebraically closed field, $R$ is a subring of $K$, $f \in K[z]$ is a monic polynomial of degree $d$ and $n \geq 1$. Then the multipliers of $f$ at its cycles with period $n$ all lie in $R$ if and only if the polynomial $M_{n}^{f}$ splits into linear factors of $R[\lambda]$.
\end{corollary}

Let us now focus on the multiplier polynomials associated with the family $\left( f_{c} \right)_{c \in \mathbb{C}}$. We may determine the multiplicities of their roots to obtain the following stronger version of Proposition~\ref{proposition:multroots}:

\begin{proposition}
\label{proposition:multrootsuni}
Assume that $c \in \mathbb{C}$ and $n \geq 1$. Then $\lambda_{0} \in \mathbb{C}$ is a root of the polynomial $M_{n}^{f_{c}}$ if and only if 
\begin{itemize}
\item either $\lambda_{0}$ is not $1$ and is the multiplier of $f_{c}$ at a cycle with period $n$, in which case $\ord_{\lambda_{0}} M_{n}^{f_{c}}$ equals the number of cycles for $f_{c}$ with period $n$ and multiplier $\lambda_{0}$,
\item or $\lambda_{0}$ equals $1$ and is the multiplier of $f_{c}$ at a cycle with period $n$, in which case $\ord_{\lambda_{0}} M_{n}^{f_{c}} = 2$,
\item or $\lambda_{0}$ equals $1$ and $f_{c}$ has a cycle with period a proper divisor $k$ of $n$ and multiplier a primitive $\frac{n}{k}$th root of unity, in which case $\ord_{\lambda_{0}} M_{n}^{f_{c}} = 1$.
\end{itemize}
\end{proposition}

\begin{proof}
Since the polynomial $\Phi_{n}^{f_{c}}$ is monic, we have \[ M_{n}^{f_{c}}(\lambda)^{n} = \prod_{j = 1}^{\nu(n)} \left( \lambda -\left( f_{c}^{\circ n} \right)^{\prime}\left( z_{j} \right) \right) \, \text{,} \] where $z_{1}, \dotsc, z_{\nu(n)}$ are the~-- not necessarily distinct~-- roots of the polynomial $\Phi_{n}^{f_{c}}$. By Proposition~\ref{proposition:dynarootsuni}, it follows that \[ M_{n}^{f_{c}}(\lambda) = (\lambda -1)^{2 p +q} \prod_{j = 1}^{r} \left( \lambda -\left( f_{c}^{\circ n} \right)^{\prime}\left( w_{j} \right) \right) \, \text{,} \] where $p$ is the number of cycles for $f_{c}$ with period $n$ and multiplier $1$, $q$ is the number of cycles for $f_{c}$ with period a proper divisor $k$ of $n$ and multiplier a primitive $\frac{n}{k}$th root of unity and $w_{1}, \dotsc, w_{r}$ are representatives for the cycles for $f_{c}$ with period $n$ and multiplier different from $1$. Since every cycle of parabolic basins for $f_{c}$ contains the unique critical point $0 \in \mathbb{C}$, we have $p +q \in \lbrace 0, 1 \rbrace$. This completes the proof of the proposition.
\end{proof}

Now, view $c$ as an indeterminate over $\mathbb{Z}$ and recall that \[ \boldsymbol{R} = \mathbb{Z}[c] \quad \text{and} \quad \boldsymbol{f}(z) = z^{d} +c \in \boldsymbol{R}[z] \, \text{.} \] For $n \geq 1$, define $M_{n} \in \mathbb{Z}[c, \lambda]$ to be the image of $M_{n}^{\boldsymbol{f}}$ under the canonical ring isomorphism from $\boldsymbol{R}[\lambda]$ to $\mathbb{Z}[c, \lambda]$.

By the uniqueness in Proposition~\ref{proposition:multdef}, for every $c \in \mathbb{C}$ and every $n \geq 1$, we have \[ M_{n}^{f_{c}}(\lambda) = M_{n}(c, \lambda) \in \mathbb{C}[\lambda] \, \text{.} \] In particular, the coefficients of the polynomials $M_{n}^{f_{c}}$, with $n \geq 1$, are polynomials in $c$ with integer coefficients. In fact, the following result shows that more is true, as observed in Remark~\ref{remark:multcoeffs2} and Remark~\ref{remark:multcoeffs3}.

\begin{proposition}
\label{proposition:multprop}
For every $n \geq 1$, the polynomial $M_{n}$ lies in $\mathbb{Z}\left[ d^{d} c^{d -1}, \lambda \right]$, has leading coefficient $\pm d^{\nu(n)}$ and degree $\frac{(d -1) \nu(n)}{d}$ in $c$ and is monic in $\lambda$ of degree $\frac{\nu(n)}{n}$.
\end{proposition}

\begin{proof}
We have \[ M_{n}(c, \lambda) = \lambda^{\frac{\nu(n)}{n}} +\sum_{j = 1}^{\frac{\nu(n)}{n}} (-1)^{j} \sigma_{j}(c) \lambda^{\frac{\nu(n)}{n} -j} \, \text{,} \] where $\sigma_{1}, \dotsc, \sigma_{\frac{\nu(n)}{n}} \in \boldsymbol{R}$ are the elementary symmetric functions of the~-- not necessarily distinct~-- roots of the polynomial $M_{n}^{\boldsymbol{f}}$ in an algebraic closure $\boldsymbol{K}$ of the field of fractions of $\boldsymbol{R}$. Moreover, for every parameter $c \in \mathbb{C}$, since we have \[ M_{n}^{f_{c}}(\lambda) = M_{n}(c, \lambda) \in \mathbb{C}[\lambda] \, \text{,} \] $\sigma_{1}(c), \dotsc, \sigma_{\frac{\nu(n)}{n}}(c)$ are the elementary symmetric functions of the~-- not necessarily distinct~-- roots of the polynomial $M_{n}^{f_{c}}$.

First, let us prove that, for every $j \in \left\lbrace 1, \dotsc, \frac{\nu(n)}{n} \right\rbrace$, we have $\sigma_{j} \in \mathbb{Z}\left[ d^{d} c^{d -1} \right]$. Choose a primitive $(d -1)$th root of unity $\omega \in \mathbb{C}$. For every parameter $c \in \mathbb{C}$, since the maps $f_{c}$ and $f_{\omega c}$ are affinely conjugate, it follows from Proposition~\ref{proposition:multrootsuni} that the polynomials $M_{n}^{f_{c}}$ and $M_{n}^{f_{\omega c}}$ have the same roots~-- counting multiplicities. Therefore, for every $j \in \left\lbrace 1, \dotsc, \frac{\nu(n)}{n} \right\rbrace$, we have $\sigma_{j}(c) = \sigma_{j}(\omega c)$ for all $c \in \mathbb{C}$, and hence $\sigma_{j}$ lies in the subring $\mathbb{Z}\left[ c^{d -1} \right]$ of $\boldsymbol{R}$. Since $\mathbb{Z}\left[ d^{d} c^{d -1} \right]$ is integrally closed in $\mathbb{Z}\left[ c^{d -1} \right]$ and each root of $M_{n}^{\boldsymbol{f}}$ in $\boldsymbol{K}$ is of the form \[ \left( \boldsymbol{f}^{\circ n} \right)^{\prime}\left( z_{0} \right) = \prod_{j = 0}^{n -1} \boldsymbol{f}^{\prime}\left( \boldsymbol{f}^{\circ j}\left( z_{0} \right) \right) \, \text{,} \] where $z_{0} \in \boldsymbol{K}$ is a periodic point for $\boldsymbol{f} \colon \boldsymbol{K} \rightarrow \boldsymbol{K}$, it suffices to prove the fact below (compare~\cite[Theorem~1.1]{M2014}) to conclude that $\sigma_{j} \in \mathbb{Z}\left[ d^{d} c^{d -1} \right]$ for all $j \in \left\lbrace 1, \dotsc, \frac{\nu(n)}{n} \right\rbrace$. Thus, the polynomial $M_{n}$ lies in $\mathbb{Z}\left[ d^{d} c^{d -1}, \lambda \right]$.

\begin{claim}
\label{claim:integral}
If $z_{0} \in \boldsymbol{K}$ is a periodic point for $\boldsymbol{f}$, then $\boldsymbol{f}^{\prime}\left( z_{0} \right)$ is integral over $\mathbb{Z}\left[ d^{d} c^{d -1} \right]$.
\end{claim}

\begin{proof}[Proof of Claim~\ref{claim:integral}]
Choose a primitive $(d -1)$th root $d^{\frac{1}{d -1}}$ of $d$ in $\boldsymbol{K}$, and, for $m \in \mathbb{Z}$, define $d^{\frac{m}{d -1}} = \left( d^{\frac{1}{d -1}} \right)^{m}$. For every $k \geq 1$, we have \[ d^{\frac{d^{k}}{d -1}} \boldsymbol{f}^{\circ k}\left( \frac{z}{d^{\frac{1}{d -1}}} \right) = \left( d^{\frac{d^{k -1}}{d -1}} \boldsymbol{f}^{\circ (k -1)}\left( \frac{z}{d^{\frac{1}{d -1}}} \right) \right)^{d} +d^{\frac{d^{k}}{d -1}} c \, \text{.} \] It follows by induction that, for every $k \geq 0$, the polynomial $d^{\frac{d^{k}}{d -1}} \boldsymbol{f}^{\circ k}\left( \frac{z}{d^{\frac{1}{d -1}}} \right)$ is monic in $z$ of degree $d^{k}$ with coefficients in $\mathbb{Z}\left[ d^{\frac{d}{d -1}} c \right]$. Therefore, $d^{\frac{1}{d -1}} z_{0}$ is integral over $\mathbb{Z}\left[ d^{d} c^{d -1} \right]$ since it is a root of the monic polynomial \[ d^{\frac{d^{k}}{d -1}} \boldsymbol{f}^{\circ k}\left( \frac{z}{d^{\frac{1}{d -1}}} \right) -d^{\frac{d^{k} -1}{d -1}} z \in \mathbb{Z}\left[ d^{\frac{d}{d -1}} c \right][z] \, \text{,} \] where $k$ is the period of $z_{0}$, and $d^{\frac{d}{d -1}} c$ is integral over $\mathbb{Z}\left[ d^{d} c^{d -1} \right]$, and hence the same is true of $\boldsymbol{f}^{\prime}\left( z_{0} \right) = d z_{0}^{d -1}$. Thus, the claim is proved.
\end{proof}

It remains to examine the leading term of the polynomial $M_{n}$ in $c$. To do this, let us first prove the following fact:

\begin{claim}
\label{claim:bound}
If $c \in \mathbb{C}$ and $z_{0} \in \mathbb{C}$ is a periodic point for $f_{c}$, then $\left\lvert z_{0} \right\rvert \leq 1 +\lvert c \rvert^{\frac{1}{d}}$.
\end{claim}

\begin{proof}[Proof of Claim~\ref{claim:bound}]
The map $\psi_{c} \colon x \mapsto x^{d} -x -\lvert c \rvert$ is strictly increasing on $[1, +\infty)$ and satisfies $\psi_{c}\left( 1 +\lvert c \rvert^{\frac{1}{d}} \right) \geq 0$. Therefore, whenever $\lvert z \rvert > 1 +\lvert c \rvert^{\frac{1}{d}}$, we have \[ \left\lvert f_{c}(z) \right\rvert \geq \lvert z \rvert^{d} -\lvert c \rvert > \lvert z \rvert \, \text{.} \] It follows that, if $z_{0} \in \mathbb{C}$ satisfies $\left\lvert z_{0} \right\rvert > 1 +\lvert c \rvert^{\frac{1}{d}}$, then $\left( f_{c}^{\circ k}\left( z_{0} \right) \right)_{k \geq 0}$ diverges to $\infty$, and hence $z_{0}$ is not periodic for $f_{c}$. Thus, the claim is proved.
\end{proof}

Now, note that, if $c \in \mathbb{C}$ and $\lambda_{0} \in \mathbb{C}$ is a root of the polynomial $M_{n}^{f_{c}}$, then there exists a periodic point $z_{0} \in \mathbb{C}$ for $f_{c}$ such that \[ \lambda_{0} = \left( f_{c}^{\circ n} \right)^{\prime}\left( z_{0} \right) = d^{n} \prod_{j = 0}^{n -1} f^{\circ j}\left( z_{0} \right)^{d -1} \, \text{.} \] If $c \in \mathbb{C}$ and $z_{0} \in \mathbb{C}$ is a periodic point for $f_{c}$, then we have $\left\lvert f_{c}\left( z_{0} \right) \right\rvert \leq 1 +\lvert c \rvert^{\frac{1}{d}}$ by Claim~\ref{claim:bound}, and hence \[ \left\lvert \left\lvert z_{0} \right\rvert -\lvert c \rvert^{\frac{1}{d}} \right\rvert \leq \left( 1 +\lvert c \rvert^{\frac{1}{d}} \right)^{\frac{1}{d}} \, \text{.} \] It follows that there exists a map $\eta \colon \mathbb{C} \rightarrow \mathbb{R}_{> 0}$ that satisfies \[ \eta(c) = O\left( \lvert c \rvert^{\frac{(d -1) n}{d} -\frac{d -1}{d^{2}}} \right) \quad \text{as} \quad c \rightarrow \infty \] and \[ \left\lvert \left\lvert \lambda_{0} \right\rvert -d^{n} \lvert c \rvert^{\frac{(d -1) n}{d}} \right\rvert \leq \eta(c) \] whenever $c \in \mathbb{C}$ and $\lambda_{0} \in \mathbb{C}$ is a root of $M_{n}^{f_{c}}$. Therefore, for every $j \in \left\lbrace 1, \dotsc, \frac{\nu(n)}{n} \right\rbrace$, the polynomial $\sigma_{j} \in \mathbb{Z}[c]$ has degree at most $\frac{j (d -1) n}{d}$, with equality and leading coefficient $\pm d^{\nu(n)}$ if $j = \frac{\nu(n)}{n}$. This completes the proof of the proposition.
\end{proof}

Now, let us consider the \emph{multibrot set} \[ \mathcal{M} = \left\lbrace c \in \mathbb{C} : 0 \text{ has bounded forward orbit under } f_{c} \right\rbrace \, \text{.} \] For $n \geq 1$, we call \emph{hyperbolic component} of $\mathcal{M}$ with period $n$ a component $W$ of the set of parameters $c \in \mathbb{C}$ for which the map $f_{c}$ has an attracting cycle with period $n$. Given a hyperbolic component $W$ of $\mathcal{M}$ and a parameter $c \in W$, we denote by $\lambda_{W}(c)$ the multiplier of $f_{c}$ at its unique attracting cycle. For every hyperbolic component $W$ of $\mathcal{M}$, the map $\lambda_{W} \colon W \rightarrow D(0, 1)$ is a branched $(d -1)$-sheeted holomorphic covering map, its unique critical point is the unique parameter $c_{W} \in W$ for which the point $0$ is periodic for $f_{c_{W}}$ and $\lambda_{W}$ extends to a unique map $\widetilde{\lambda_{W}} \colon \overline{W} \rightarrow \overline{D(0, 1)}$ that induces a covering map from $\overline{W} \setminus \left\lbrace c_{W} \right\rbrace$ onto $\overline{D(0, 1)} \setminus \lbrace 0 \rbrace$. Furthermore, for every $n \geq 1$, the hyperbolic components of $\mathcal{M}$ with period $n$ have pairwise disjoint closures (see~\cite[Expos\'{e}~XIV and Expos\'{e}~XIX]{DH1985} or~\cite[Theorem~6.5]{M2000}). Using these facts, we obtain the following result (compare~\cite[Proposition~3.2]{MV1995}):

\begin{proposition}
\label{proposition:multsep}
Assume that $n \geq 1$. Then, for every $\lambda_{0} \in \mathbb{C}$ that satisfies $0 < \left\lvert \lambda_{0} \right\rvert \leq 1$, the polynomial $M_{n}\left( c, \lambda_{0} \right) \in \mathbb{C}[c]$ is separable. Furthermore, we have \[ M_{n}(c, 0) = \pm d^{\nu(n)} \Phi_{n}(c, 0)^{d -1} \] and the polynomial $\Phi_{n}(c, 0) \in \mathbb{Z}[c]$ is separable.
\end{proposition}

Suppose that $n \geq 1$. The polynomial $\Phi_{n}^{\boldsymbol{f}}$ is irreducible over $\mathbb{C}(c)$ by Proposition~\ref{proposition:dynairr}, and hence its roots in an algebraic closure of $\mathbb{C}(c)$ are Galois conjugates of each other over $\mathbb{C}(c)$. It follows that the same is true of the polynomial $M_{n}^{\boldsymbol{f}}$, which shows that the polynomial $M_{n}$ is the power of some irreducible polynomial in $\mathbb{C}[c, \lambda]$. Therefore, since the polynomial $M_{n}(c, 1) \in \mathbb{Z}[c]$ is separable by Proposition~\ref{proposition:multsep}, we have the following:

\begin{proposition}[{\cite[Corollary~1]{M1996}}]
For every $n \geq 1$, the polynomial $M_{n}$ is irreducible over $\mathbb{C}$.
\end{proposition}

\subsection{Discriminants of the multiplier polynomials}

Finally, let us study here the discriminants of the multiplier polynomials associated with the family $\left( f_{c} \right)_{c \in \mathbb{C}}$. Their expressions, for $d \in \lbrace 2, 3 \rbrace$ and small values of $n$, are an essential ingredient in our proofs of Proposition~\ref{proposition:quadratic} and Proposition~\ref{proposition:cubic}.

For $n \geq 1$, define \[ \Delta_{n}(c) = \disc_{\lambda} M_{n}(c, \lambda) \in \mathbb{Z}[c] \, \text{.} \] In fact, for every $n \geq 1$, the polynomial $\Delta_{n}$ lies in $\mathbb{Z}\left[ d^{d} c^{d -1} \right]$ by Proposition~\ref{proposition:multprop}. Furthermore, for every $c \in \mathbb{C}$ and every $n \geq 1$, we have \[ \Delta_{n}(c) = \disc M_{n}^{f_{c}} \in \mathbb{C} \, \text{.} \]

\begin{example}
\label{example:delta1}
By Example~\ref{example:mult1}, we have \[ M_{1}(c, \lambda) = \lambda (\lambda -d)^{d -1} +(-d)^{d} c^{d -1} \quad \text{and} \quad \frac{\partial M_{1}}{\partial \lambda}(c, \lambda) = d (\lambda -1) (\lambda -d)^{d -2} \, \text{.} \] Therefore, we have \[ \begin{split} \Delta_{1}(c) & = (-1)^{\frac{d (d -1)}{2}} \res_{\lambda}\left( M_{1}(c, \lambda), \frac{\partial M_{1}}{\partial \lambda}(c, \lambda) \right)\\ & = (-1)^{\frac{d (d -1)}{2}} d^{d} M_{1}(c, 1) M_{1}(c, d)^{d -2}\\ & = (-1)^{\frac{d (d -1)}{2}} d^{d (d -1)} c^{(d -1) (d -2)} \left( d^{d} c^{d -1} -(d -1)^{d -1} \right) \, \text{.} \end{split} \]
\end{example}

Suppose that $n \geq 1$. By Proposition~\ref{proposition:multrootsuni}, the roots of the polynomial $\Delta_{n}$ are precisely the parameters $c_{0} \in \mathbb{C}$ for which $f_{c_{0}}$ has a cycle with period $n$ and multiplier $1$ or $f_{c_{0}}$ has two distinct cycles with period $n$ and the same multiplier. Thus, in order to factor $\Delta_{n}$, it is natural to try to define a polynomial that vanishes precisely at the parameters $c_{0} \in \mathbb{C}$ for which the map $f_{c_{0}}$ has a cycle with period $n$ and multiplier $1$. Note that, by Proposition~\ref{proposition:multrootsuni}, the roots of the polynomial $M_{n}(c, 1) \in \mathbb{Z}[c]$ are precisely the parameters $c_{0} \in \mathbb{C}$ for which either $f_{c_{0}}$ has a cycle with period $n$ and multiplier $1$ or $f_{c_{0}}$ has a cycle with period a proper divisor $k$ of $n$ and multiplier a primitive $\frac{n}{k}$th root of unity, which suggests factoring this polynomial.

For $k \geq 1$ and $l \geq 2$, define \[ P_{k, l}(c) = \res_{\lambda}\left( C_{l}(\lambda), M_{k}(c, \lambda) \right) \in \mathbb{Z}[c] \, \text{,} \] where $C_{l} \in \mathbb{Z}[\lambda]$ denotes the $l$th cyclotomic polynomial. We have \[ P_{k, l}(c) = \prod_{j = 1}^{\varphi(l)} M_{k}\left( c, \omega_{j} \right) \, \text{,} \] where $\varphi \colon \mathbb{Z}_{\geq 1} \rightarrow \mathbb{Z}_{\geq 1}$ denotes Euler's totient function and $\omega_{1}, \dotsc, \omega_{\varphi(l)}$ are the primitive $l$th roots of unity in $\mathbb{C}$. By Proposition~\ref{proposition:multrootsuni} and Proposition~\ref{proposition:multprop}, it follows that the polynomial $P_{k, l}$ lies in $\mathbb{Z}\left[ d^{d} c^{d -1} \right]$, has leading coefficient $\pm 1$ in $d^{d} c^{d -1}$ and its roots are precisely the parameters $c_{0} \in \mathbb{C}$ for which the map $f_{c_{0}}$ has a cycle with period $k$ and multiplier a primitive $l$th root of unity. Furthermore, by Proposition~\ref{proposition:multsep} and since every cycle of parabolic basins contains a critical point, the polynomial $P_{k, l}$ is separable and, for every $k^{\prime} \geq 1$ and every $l^{\prime} \geq 2$ such that $(k, l) \neq \left( k^{\prime}, l^{\prime} \right)$, the polynomials $P_{k, l}$ and $P_{k^{\prime}, l^{\prime}}$ have no common roots.

\begin{remark}
For every $k \geq 1$ and every $l \geq 2$, the parameters $c_{0} \in \mathbb{C}$ for which the map $f_{c_{0}}$ has a cycle with period $k$ and multiplier a primitive $l$th root of unity are precisely the parameters at which a cycle with period $k l$ degenerates to a cycle with period $k$. These are the parameters that lie in the intersections of the closures of hyperbolic components of $\mathcal{M}$ with period $k$ with the closures of hyperbolic components of $\mathcal{M}$ with period $k l$.
\end{remark}

Suppose that $n \geq 1$. By Proposition~\ref{proposition:multrootsuni}, Proposition~\ref{proposition:multprop} and Proposition~\ref{proposition:multsep}, the polynomial $M_{n}(c, 1) \in \mathbb{Z}[c]$ lies in $\mathbb{Z}\left[ d^{d} c^{d -1} \right]$, has leading coefficient $\pm 1$ in $d^{d} c^{d -1}$ and its roots are simple and are precisely the roots of the polynomials $P_{k, \frac{n}{k}}$, with $k$ a proper divisor of $n$, and the parameters $c_{0} \in \mathbb{C}$ for which the map $f_{c_{0}}$ has a cycle with period $n$ and multiplier $1$. Furthermore, the polynomial \[ \prod_{k \mid n, \, k \neq n} P_{k, \frac{n}{k}} \in \mathbb{Z}[c] \] lies in $\mathbb{Z}\left[ d^{d} c^{d -1} \right]$, has leading coefficient $\pm 1$ in $d^{d} c^{d -1}$ and is separable by the discussion above. Therefore, there exists a unique polynomial $Q_{n} \in \mathbb{Z}[c]$ such that \[ M_{n}(c, 1) = Q_{n}(c) \prod_{k \mid n, \, k \neq n} P_{k, \frac{n}{k}}(c) \] and the polynomial $Q_{n}$ lies in $\mathbb{Z}\left[ d^{d} c^{d -1} \right]$, has leading coefficient $\pm 1$ in $d^{d} c^{d -1}$ and its roots are simple and are precisely the parameters $c_{0} \in \mathbb{C}$ for which the map $f_{c_{0}}$ has a cycle with period $n$ and multiplier $1$ (see~\cite[Theorem~A]{MV1995}).

\begin{remark}
For every $n \geq 1$, the parameters $c_{0} \in \mathbb{C}$ for which the map $f_{c_{0}}$ has a cycle with period $n$ and multiplier $1$ are precisely the parameters at which two distinct cycles with period $n$ collide. These are the parameters that occur at the cusps of the hyperbolic components of $\mathcal{M}$ with period $n$.
\end{remark}

Since the polynomials $Q_{n}$ and $\Delta_{n}$ lie in $\mathbb{Z}\left[ d^{d} c^{d -1} \right]$ and the polynomial $Q_{n}$ has leading coefficient $\pm 1$ in $d^{d} c^{d -1}$ and its roots are simple and are also roots of $\Delta_{n}$, the polynomial $Q_{n}$ divides $\Delta_{n}$ in $\mathbb{Z}\left[ d^{d} c^{d -1} \right]$. In fact, the result below shows that more is true, as observed in Remark~\ref{remark:deltafact} (compare~\cite[Proposition~9]{M1996}).

\begin{proposition}
For every $n \geq 1$, there exist a unique squarefree integer $a_{n}$ and a unique polynomial $R_{n} \in \mathbb{Z}[c]$ with positive leading coefficient that satisfy \[ \Delta_{n} = a_{n} Q_{n} R_{n}^{2} \, \text{.} \] Furthermore, the polynomial $R_{n}$ lies in $\mathbb{Z}\left[ d^{d} c^{d -1} \right]$ and its roots are precisely the parameters $c_{0} \in \mathbb{C}$ for which the map $f_{c_{0}}$ has two distinct cycles with period $n$ and the same multiplier.
\end{proposition}

\begin{proof}
Since the polynomials $Q_{n}$ and $\Delta_{n}$ lie in $\mathbb{Z}\left[ d^{d} c^{d -1} \right]$ and the polynomial $Q_{n}$ has leading coefficient $\pm 1$ in $d^{d} c^{d -1}$, is separable and its roots are also roots of $\Delta_{n}$, it suffices to prove that, for every parameter $c_{0} \in \mathbb{C}$, we have \[ \ord_{c_{0}} \Delta_{n} = \varepsilon_{n}\left( c_{0} \right) +2 \kappa_{n}\left( c_{0} \right) \, \text{,} \] where $\varepsilon_{n}\left( c_{0} \right) \in \lbrace 0, 1 \rbrace$ equals $1$ if and only if $f_{c_{0}}$ has a cycle with period $n$ and multiplier $1$ and $\kappa_{n}\left( c_{0} \right) \in \mathbb{Z}_{\geq 0}$ is positive if and only if $f_{c_{0}}$ has two distinct cycles with period $n$ and the same multiplier.

Suppose that $c_{0} \in \mathbb{C}$. Choose representatives $z_{1}, \dotsc, z_{r}$ for the cycles for $f_{c_{0}}$ with period $n$ and multiplier different from $1$. For every $j \in \lbrace 1, \dotsc, r \rbrace$, we have $\ord_{z_{j}} \Phi_{n}^{f_{c_{0}}} = 1$ by Proposition~\ref{proposition:dynarootsuni}, and hence $\frac{\partial \Phi_{n}}{\partial z}\left( c_{0}, z_{j} \right) \neq 0$. By the implicit function theorem, it follows that there exist a complex domain $U$ containing $c_{0}$ and holomorphic maps \[ \zeta_{1}, \dotsc, \zeta_{r} \colon U \rightarrow \mathbb{C} \] such that, for every $j \in \lbrace 1, \dotsc, r \rbrace$, we have $\zeta_{j}\left( c_{0} \right) = z_{j}$ and $\Phi_{n}\left( c, \zeta_{j}(c) \right) = 0$ for all $c \in U$. Now, embed $\boldsymbol{R} = \mathbb{Z}[c]$ into the ring $\boldsymbol{S} = \mathcal{H}(U)$ of holomorphic maps on $U$. For every $j \in \lbrace 1, \dotsc, r \rbrace$, we have $\Phi_{n}^{\boldsymbol{f}}\left( \zeta_{j} \right) = 0$, and hence $\zeta_{j}$ is a periodic point for the map $\boldsymbol{f} \colon \boldsymbol{S} \rightarrow \boldsymbol{S}$ with period a divisor of $n$. Since the points $\zeta_{j}\left( c_{0} \right) = z_{j}$, with $j \in \lbrace 1, \dotsc, r \rbrace$, are periodic for $f_{c_{0}}$ with period $n$ and belong to pairwise distinct cycles, it follows that the points $\zeta_{j}$, with $j \in \lbrace 1, \dotsc, r \rbrace$, are periodic points for $\boldsymbol{f} \colon \boldsymbol{S} \rightarrow \boldsymbol{S}$ with period $n$, which belong to pairwise distinct cycles. Consequently, there exists a unique monic polynomial $\Psi_{n} \in \boldsymbol{S}[z]$ such that \[ \Phi_{n}^{\boldsymbol{f}}(z) = \Psi_{n}(z) \prod_{j = 1}^{r} \prod_{k = 0}^{n -1} \left( z -\boldsymbol{f}^{\circ k}\left( \zeta_{j} \right) \right) \, \text{,} \] and we have \[ M_{n}^{\boldsymbol{f}}(\lambda)^{n} = \res_{z}\left( \Psi_{n}(z), \lambda -\left( \boldsymbol{f}^{\circ n} \right)^{\prime}(z) \right) \prod_{j = 1}^{r} \left( \lambda -\lambda_{j} \right)^{n} \, \text{,} \] where $\lambda_{j} = \left( \boldsymbol{f}^{\circ n} \right)^{\prime}\left( \zeta_{j} \right)$ is the multiplier of $\boldsymbol{f} \colon \boldsymbol{S} \rightarrow \boldsymbol{S}$ at $\zeta_{j}$ for $j \in \lbrace 1, \dotsc, r \rbrace$. Therefore, there exists a unique monic polynomial $N_{n} \in \boldsymbol{S}[\lambda]$ that satisfies \[ M_{n}^{\boldsymbol{f}}(\lambda) = N_{n}(\lambda) \prod_{j = 1}^{r} \left( \lambda -\lambda_{j} \right) \, \text{.} \] Now, define \[ \kappa_{n}\left( c_{0} \right) = \sum_{1 \leq j < k \leq r} \ord_{c_{0}}\left( \lambda_{j} -\lambda_{k} \right) \in \mathbb{Z}_{\geq 0} \, \text{,} \] which is positive if and only if the map $f_{c_{0}}$ has two distinct cycles with period $n$ and the same multiplier since $\lambda_{j}\left( c_{0} \right) = \left( f_{c_{0}}^{\circ n} \right)^{\prime}\left( z_{j} \right)$ for all $j \in \lbrace 1, \dotsc, r \rbrace$. Let us consider three different cases.

Assume that $f_{c_{0}}$ has neither a cycle with period a proper divisor $k$ of $n$ and multiplier a primitive $\frac{n}{k}$th root of unity nor a cycle with period $n$ and multiplier $1$. Then, by Proposition~\ref{proposition:multrootsuni}, we have \[ M_{n}^{f_{c_{0}}}(\lambda) = \prod_{j = 1}^{r} \left( \lambda -\left( f_{c_{0}}^{\circ n} \right)^{\prime}\left( z_{j} \right) \right) = \prod_{j = 1}^{r} \left( \lambda -\lambda_{j}\left( c_{0} \right) \right) \, \text{,} \] and hence $r = \frac{\nu(n)}{n}$ and $N_{n} = 1$. Therefore, we have \[ \Delta_{n} = \disc M_{n}^{\boldsymbol{f}} = \prod_{1 \leq j < k \leq \frac{\nu(n)}{n}} \left( \lambda_{j} -\lambda_{k} \right)^{2} \, \text{,} \] and hence $\ord_{c_{0}} \Delta_{n} = 2 \kappa_{n}\left( c_{0} \right)$.

Assume now that $f_{c_{0}}$ has a cycle with period a proper divisor $k$ of $n$ and multiplier a primitive $\frac{n}{k}$th root of unity. Then, by Proposition~\ref{proposition:multrootsuni}, we have \[ M_{n}^{f_{c_{0}}}(\lambda) = (\lambda -1) \prod_{j = 1}^{r} \left( \lambda -\lambda_{j}\left( c_{0} \right) \right) \, \text{,} \] and hence $r = \frac{\nu(n)}{n} -1$ and there exists a holomorphic map $\rho \colon U \rightarrow \mathbb{C}$ such that $\rho\left( c_{0} \right) = 1$ and $N_{n}(\lambda) = \lambda -\rho$. Therefore, we have \[ \Delta_{n} = \prod_{j = 1}^{\frac{\nu(n)}{n} -1} \left( \rho -\lambda_{j} \right)^{2} \prod_{1 \leq j < k \leq \frac{\nu(n)}{n} -1} \left( \lambda_{j} -\lambda_{k} \right)^{2} \, \text{,} \] and hence $\ord_{c_{0}} \Delta_{n} = 2 \kappa_{n}\left( c_{0} \right)$ since \[ \prod_{j = 1}^{\frac{\nu(n)}{n} -1} \left( \rho -\lambda_{j} \right)\left( c_{0} \right)^{2} = \prod_{j = 1}^{\frac{\nu(n)}{n} -1} \left( 1 -\left( f_{c_{0}}^{\circ n} \right)^{\prime}\left( z_{j} \right) \right)^{2} \neq 0 \, \text{.} \]

Finally, assume that $f_{c_{0}}$ has a cycle with period $n$ and multiplier $1$. Then, by Proposition~\ref{proposition:multrootsuni}, we have \[ M_{n}^{f_{c_{0}}}(\lambda) = (\lambda -1)^{2} \prod_{j = 1}^{r} \left( \lambda -\lambda_{j}\left( c_{0} \right) \right) \, \text{,} \] and hence $r = \frac{\nu(n)}{n} -2$ and there exist holomorphic maps $\sigma_{1}, \sigma_{2} \colon U \rightarrow \mathbb{C}$ such that \[ \sigma_{1}\left( c_{0} \right) = 2 \, \text{,} \quad \sigma_{2}\left( c_{0} \right) = 1 \quad \text{and} \quad N_{n}(\lambda) = \lambda^{2} -\sigma_{1} \lambda +\sigma_{2} \, \text{.} \] Therefore, we have \[ \Delta_{n} = \left( \sigma_{1}^{2} -4 \sigma_{2} \right) \prod_{j = 1}^{\frac{\nu(n)}{n} -2} \left( \lambda_{j}^{2} -\sigma_{1} \lambda_{j} +\sigma_{2} \right)^{2} \prod_{1 \leq j < k \leq \frac{\nu(n)}{n} -2} \left( \lambda_{j} -\lambda_{k} \right)^{2} \, \text{,} \] and hence \[ \ord_{c_{0}} \Delta_{n} = \ord_{c_{0}}\left( \sigma_{1}^{2} -4 \sigma_{2} \right) +2 \kappa_{n}\left( c_{0} \right) \] since we have \[ \prod_{j = 1}^{\frac{\nu(n)}{n} -2} \left( \lambda_{j}^{2} -\sigma_{1} \lambda_{j} +\sigma_{2} \right)\left( c_{0} \right)^{2} = \prod_{j = 1}^{\frac{\nu(n)}{n} -2} \left( \left( f_{c_{0}}^{\circ n} \right)^{\prime}\left( z_{j} \right) -1 \right)^{4} \neq 0 \, \text{.} \] We have $\left( \sigma_{1}^{2} -4 \sigma_{2} \right)\left( c_{0} \right) = 0$. Since the polynomial $M_{n}(c, 1) \in \mathbb{Z}[c]$ is separable by Proposition~\ref{proposition:multsep}, we have \[ \frac{\partial M_{n}}{\partial c}\left( c_{0}, 1 \right) = \left( -\sigma_{1}^{\prime}\left( c_{0} \right) +\sigma_{2}^{\prime}\left( c_{0} \right) \right) \prod_{j = 1}^{\frac{\nu(n)}{n} -2} \left( 1 -\lambda_{j}\left( c_{0} \right) \right) \neq 0 \, \text{,} \] and hence \[ \left( \sigma_{1}^{2} -4 \sigma_{2} \right)^{\prime}\left( c_{0} \right) = -4 \left( -\sigma_{1}^{\prime}\left( c_{0} \right) +\sigma_{2}^{\prime}\left( c_{0} \right) \right) \neq 0 \, \text{.} \] Therefore, we have $\ord_{c_{0}}\left( \sigma_{1}^{2} -4 \sigma_{2} \right) = 1$, and hence \[ \ord_{c_{0}} \Delta_{n} = 1 +2 \kappa_{n}\left( c_{0} \right) \, \text{.} \] This completes the proof of the proposition.
\end{proof}

\begin{example}
\label{example:delta1fact}
By Example~\ref{example:mult1} and Example~\ref{example:delta1}, we have \[ Q_{1}(c) = M_{1}(c, 1) = (-1)^{d} \left( d^{d} c^{d -1} -(d -1)^{d -1} \right) \] and \[ \Delta_{1}(c) = (-1)^{\frac{d (d -1)}{2}} d^{d (d -1)} c^{(d -1) (d -2)} \left( d^{d} c^{d -1} -(d -1)^{d -1} \right) \, \text{.} \] Therefore, we have \[ a_{1} = (-1)^{\frac{d (d +1)}{2}} \quad \text{and} \quad R_{1}(c) = d^{\frac{d (d -1)}{2}} c^{\frac{(d -1) (d -2)}{2}} \, \text{.} \]
\end{example}

Note that the polynomials $Q_{1}$ and $R_{1}$ have no common roots, which shows that there is no parameter $c_{0} \in \mathbb{C}$ for which the map $f_{c_{0}}$ has both a fixed point with multiplier $1$ and two distinct fixed points with the same multiplier. Using the software SageMath, we observe that the same is true of the polynomials $Q_{n}$ and $R_{n}$ for small values of $d$ and $n$. Thus, it seems likely that the following question has a negative answer.

\begin{question}
Does there exist an integer $n \geq 1$ such that the polynomials $Q_{n}$ and $R_{n}$ have a common root? Equivalently, does there exist an integer $n \geq 1$ and a unicritical polynomial map $f \colon \mathbb{C} \rightarrow \mathbb{C}$ of degree $d$ that has both a cycle with period $n$ and multiplier $1$ and two distinct cycles with period $n$ and the same multiplier?
\end{question}

Finally, note that, if $n \geq 1$ and $c_{0} \in \mathbb{C}$ is a parameter such that the map $f_{c_{0}}$ has a rational multiplier at each cycle with period $1$ or $n$, then $c_{0}^{d -1}$ is rational and $\Delta_{n}\left( c_{0} \right)$ is the square of a rational number, and hence $R_{n}\left( c_{0} \right) = 0$ or $a_{n} Q_{n}\left( c_{0} \right)$ is the square of a rational number. Thus, it would be interesting to determine the integers $a_{n}$, with $n \geq 1$. We proved in Example~\ref{example:delta1fact} that $a_{1} = \pm 1$. Using the software SageMath, we also obtain that $a_{n} = \pm 1$ for small values of $d$ and $n$, which suggests the following:

\begin{question}
Do we have $a_{n} = \pm 1$ for all $n \geq 1$?
\end{question}

The periodic points for $f_{-t^{d}} \colon z \mapsto z^{d} -t^{d}$ have Laurent expansions in $t^{-(d -1)}$ with coefficients in $\mathbb{Q}(\omega)$, where $\omega \in \mathbb{C}$ is a primitive $d$th root of unity, for $t$ around $\infty$ (compare~\cite[Lemma~2]{M1996}), and hence the same is true of their multipliers. Using this fact, it is possible to prove that the question above has a positive answer when $d = 2$.

\appendix

\section{A Fermat-like Diophantine equation}
\label{section:fermat}

We shall prove here the following statement, which is a crucial argument in our proof of Proposition~\ref{proposition:quadratic}:

\begin{lemma}
\label{lemma:fermat}
Assume that $x, y, z \in \mathbb{Z}$ satisfy $x^{3} +y^{3} = 4 z^{3}$. Then $z = 0$.
\end{lemma}

Our proof of Lemma~\ref{lemma:fermat} is adapted from the proof of Fermat's last theorem for exponent $3$ presented in~\cite{H2011} and uses the principle of infinite descent.

Define $\mathbb{A} = \mathbb{Z}[j]$ to be the ring of Eisenstein integers, where $j = \exp\left( \frac{2 \pi i}{3} \right) \in \mathbb{C}$. The ring $\mathbb{A}$ is a Euclidean domain and its group of units \[ \mathbb{A}^{\times} = \left\lbrace \pm 1, \pm j, \pm j^{2} \right\rbrace \] consists of the $6$th roots of unity.

First, observe that $\lambda = 1 -j \in \mathbb{A}$ is irreducible.

\begin{claim}
\label{claim:alambda}
The quotient ring $\mathbb{A}/\lambda \mathbb{A}$ consists of the residue classes of $-1$, $0$ and $1$ (see Figure~\ref{figure:alambda}).
\end{claim}

\begin{figure}
\framebox{\begin{tikzpicture}[scale=1.25]
\clip (-4,-3) rectangle (4,3);
\fill[Ivory] (-5,-4) rectangle (5,4);
\draw[SlateBlue,line width=0.5pt] (0,0) circle (1);
\foreach \x in {-12,-9,...,12}
{\draw[Coral,line width=1pt] (\x,0) -- +(-30:-9) -- +(-30:9) -- +(30:-9) -- +(30:9);}
\foreach \x in {-8,-7,...,8}
{\foreach \y in {-5,-4,...,5}
{\path (\x,0) ++(120:\y) node[normal=Fuchsia]{};}}
\foreach \x in {-12,-9,...,12}
{\foreach \y in {-5,-4,...,5}
{\path (\x,0) ++(\y,0) ++(120:-\y) node[special={FireBrick}{Fuchsia}]{};}}
\foreach \t in {-180,-120,...,120}
{\path (\t:1) node[special={Aqua}{Fuchsia}]{};}
\path (0,0) node[below]{$0$};
\path (0:1) node[right]{$1$};
\path (120:1) node[above left]{$j$};
\path (-120:1) node[below left]{$j^{2}$};
\path (1,0) ++(120:-1) node[below]{$\lambda$};
\end{tikzpicture}}
\caption{The ring of Eisenstein integers, its units and its ideal $\lambda \mathbb{A}$.}
\label{figure:alambda}
\end{figure}
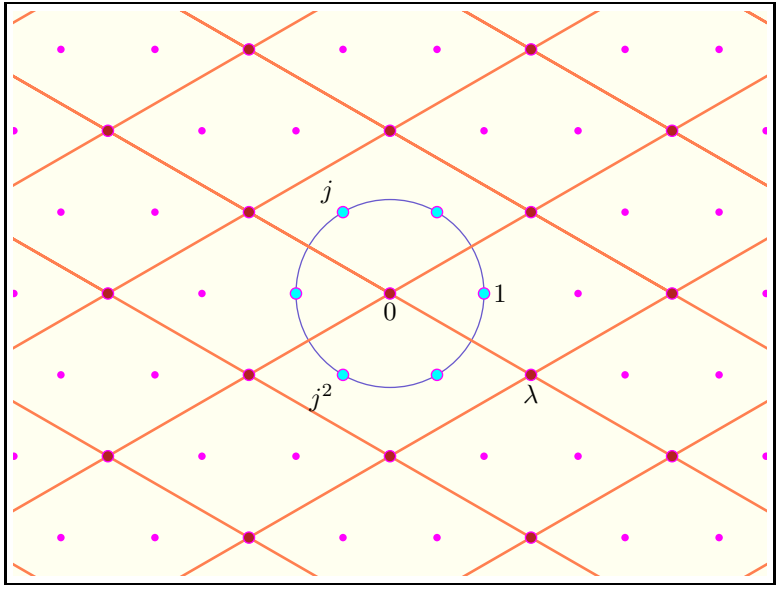

\begin{proof}[Proof of Claim~\ref{claim:alambda}]
Every element of $\mathbb{A}$ is congruent to an element of $\mathbb{Z}$ modulo $\lambda$ since $j \equiv 1 \pmod{\lambda}$ and every element of $\mathbb{Z}$ is congruent to either $-1$, $0$ or $1$ modulo $3 = -j^{2} \lambda^{2}$.
\end{proof}

\begin{claim}
\label{claim:alambda3}
The ring $\mathbb{A}/\lambda^{3} \mathbb{A}$ contains exactly $3$ cubes, namely the residue classes of $-1$, $0$ and $1$ (see Figure~\ref{figure:alambda3}). More precisely, for every $a \in \mathbb{A}$, 
\begin{itemize}
\item either $a \equiv -1 \pmod{\lambda}$ and $a^{3} \equiv -1 \pmod{\lambda^{3}}$,
\item or $a \equiv 0 \pmod{\lambda}$ and $a^{3} \equiv 0 \pmod{\lambda^{3}}$,
\item or $a \equiv 1 \pmod{\lambda}$ and $a^{3} \equiv 1 \pmod{\lambda^{3}}$.
\end{itemize}
\end{claim}

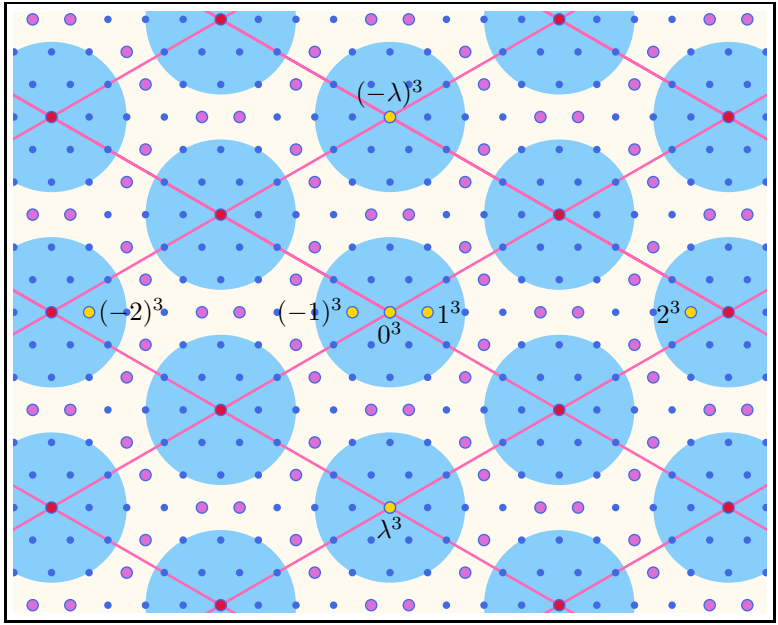
\begin{figure}
\framebox{\begin{tikzpicture}[scale=0.5]
\clip (-10,-8) rectangle (10,8);
\fill[FloralWhite] (-11,-9) rectangle (11,9);
\foreach \x in {-27,-18,...,27}
{\foreach \y in {-12,-9,...,12}
{\fill[LightSkyBlue] (\x,0) ++(\y,0) ++(120:-\y) circle (2);}}
\foreach \x in {-27,-18,...,27}
{\draw[HotPink,line width=1pt] (\x,0) -- +(-30:-19) -- +(-30:19) -- +(30:-19) -- +(30:19);}
\foreach \x in {-17,-16,...,17}
{\foreach \y in {-11,-10,...,11}
{\path (\x,0) ++(120:\y) node[normal=RoyalBlue]{};}}
\foreach \x in {-27,-18,...,27}
{\foreach \y in {-12,-9,...,12}
{\path (\x,0) ++(\y,0) ++(120:-\y) node[special={Crimson}{RoyalBlue}]{};
\foreach \t in {-180,-120,...,120}
{\path (\x,0) ++(\y,0) ++(120:-\y) ++(\t:4) node[special={Orchid}{RoyalBlue}]{};}}}
\path (-8,0) node[special={Gold}{RoyalBlue}]{} node[right]{$(-2)^{3}$};
\path (-1,0) node[special={Gold}{RoyalBlue}]{} node[left]{$(-1)^{3}$};
\path (0,0) node[special={Gold}{RoyalBlue}]{} node[below]{$0^{3}$};
\path (1,0) node[special={Gold}{RoyalBlue}]{} node[right]{$1^{3}$};
\path (8,0) node[special={Gold}{RoyalBlue}]{} node[left]{$2^{3}$};
\path (-3,0) ++(120:-6) node[special={Gold}{RoyalBlue}]{} node[below]{$\lambda^{3}$};
\path (3,0) ++(120:6) node[special={Gold}{RoyalBlue}]{} node[above]{$(-\lambda)^{3}$};
\end{tikzpicture}}
\caption{The ring of Eisenstein integers, its cubes and its ideal $\lambda^{3} \mathbb{A}$. The elements of the form $x^{3} +u y^{3}$, with $x, y \in \mathbb{A}$ and $u \in \mathbb{A}^{\times}$, lie in blue disks. The set $4 \mathbb{A}^{\times} +\lambda^{3} \mathbb{A}$ is indicated by purple dots.}
\label{figure:alambda3}
\end{figure}

\begin{proof}[Proof of Claim~\ref{claim:alambda3}]
If $\lambda$ divides $a$, then $\lambda^{3}$ divides $a^{3}$. If $a \equiv 1 \pmod{\lambda}$, then there exists $b \in \mathbb{A}$ such that $a = 1 +\lambda b$, and we have \[ a^{3} -1 = (a -1) (a -j) \left( a -j^{2} \right) = \lambda^{3} b (b +1) (b +1 +j) \, \text{.} \] If $a \equiv -1 \pmod{\lambda}$, then $a^{3} \equiv -1 \pmod{\lambda^{3}}$ since $a^{3} = -(-a)^{3}$.
\end{proof}

Finally, Lemma~\ref{lemma:fermat} follows immediately from the following more general result:

\begin{lemma}
\label{lemma:fermatgen}
Suppose that $x, y, z \in \mathbb{A}$ and $u, v \in \mathbb{A}^{\times}$ are such that $x^{3} +u y^{3} = 4 v z^{3}$. Then $z = 0$.
\end{lemma}

\begin{proof}[Proof of Lemma~\ref{lemma:fermatgen}]
To obtain a contradiction, suppose that there exist $x, y, z \in \mathbb{A}$, with $z \neq 0$, that are relatively prime and $u, v \in \mathbb{A}^{\times}$ such that $x^{3} +u y^{3} = 4 v z^{3}$ and the valuation $\ord_{\lambda}(z)$ is minimal. Then $x$, $y$ and $z$ are pairwise relatively prime. Note that, by Claim~\ref{claim:alambda3}, $x^{3} +u y^{3}$ is at distance at most $2$ from $\lambda^{3} \mathbb{A}$, while the distance between $4 \mathbb{A}^{\times}$ and $\lambda^{3} \mathbb{A}$ equals $\sqrt{7} > 2$. Therefore, we have $z \equiv 0 \pmod{\lambda}$, $u = \pm 1$ and $x \equiv -u y \pmod{\lambda}$. It follows that \[ x +u y \equiv j x +j^{2} u y \equiv j^{2} x +j u y \equiv 0 \pmod{\lambda} \, \text{,} \] and hence there exist $a, b, c, d \in \mathbb{A}$ such that \[ x +u y = \lambda a \, \text{,} \quad j x +j^{2} u y = \lambda b \, \text{,} \quad j^{2} x +j u y = \lambda c \quad \text{and} \quad z = \lambda d \, \text{.} \] We have $a +b +c = 0$ and $a b c = 4 v d^{3}$. Moreover, since \[ x = -j a +j^{2} b = a -j^{2} c = -b +j c \quad \text{and} \quad u y = a -j^{2} b = -j a +j^{2} c = j b -c \, \text{,} \] $a$, $b$ and $c$ are pairwise relatively prime. Therefore, there exist $X, Y, Z \in \mathbb{A}$, which are necessarily pairwise relatively prime, $u_{X}, u_{Y}, u_{Z} \in \mathbb{A}^{\times}$ and a permutation $\sigma$ of $\lbrace a, b, c \rbrace$ that satisfy \[ \sigma(a) = u_{X} X^{3} \, \text{,} \quad \sigma(b) = u_{Y} Y^{3} \quad \text{and} \quad \sigma(c) = 4 u_{Z} Z^{3} \, \text{.} \] We have \[ X^{3} +U Y^{3} = 4 V Z^{3} \, \text{,} \] where $U = u_{X}^{-1} u_{Y} \in \mathbb{A}^{\times}$ and $V = -u_{X}^{-1} u_{Z} \in \mathbb{A}^{\times}$, and \[ 3 \ord_{\lambda}(Z) = \ord_{\lambda}\left( \sigma(c) \right) = \ord_{\lambda}(a b c) = \ord_{\lambda}\left( d^{3} \right) = 3 \ord_{\lambda}(z) -3 \, \text{.} \] This contradicts the minimality of $\ord_{\lambda}(z)$, and thus the lemma is proved.
\end{proof}

\newcommand{\etalchar}[1]{$^{#1}$}
\providecommand{\bysame}{\leavevmode\hbox to3em{\hrulefill}\thinspace}
\providecommand{\MR}{\relax\ifhmode\unskip\space\fi MR }
\providecommand{\MRhref}[2]{%
\href{http://www.ams.org/mathscinet-getitem?mr=#1}{#2}
}
\providecommand{\href}[2]{#2}

\end{document}